\newif\ifexternalize
\externalizefalse
\newif\ifCompileAll
\CompileAlltrue

\documentclass[10pt,a4paper]{article}
\usepackage{comment,environ}
\usepackage{amsthm,amsfonts,amsmath,amssymb,latexsym}
\usepackage{mathtools}
\usepackage{epsfig,graphics,color}
\usepackage[T1]{fontenc}
\usepackage[utf8]{inputenc}
\usepackage{pifont}
\usepackage[matrix,arrow,curve]{xy}
\usepackage[breaklinks=true]{hyperref}
\usepackage{tikz}
\usetikzlibrary{cd,topaths}
\usetikzlibrary{decorations.markings, decorations.pathreplacing}
\usetikzlibrary{arrows,arrows.meta,bending}
\usetikzlibrary{positioning}
\usetikzlibrary{calc}

\ifexternalize
  \usetikzlibrary{external}
  \let\originaltikzcd=\tikzcd
  \let\originalendtikzcd=\endtikzcd
  \def\tikzcd{\tikzexternaldisable\originaltikzcd}
  \def\endtikzcd{\originalendtikzcd\tikzexternalenable}
\fi
  
\usepackage{pdfrender}
\usepackage{bm}
\usepackage{enumitem}

\usepackage[style=alphabetic]{biblatex}
\newtheorem{PARA}{}[section]
\newtheorem{theorem}[PARA]{Theorem}

\newtheorem{lemma}[PARA]{Lemma}
\newtheorem{proposition}[PARA]{Proposition}
\newtheorem{definition}[PARA]{Definition}

\theoremstyle{definition}
\newtheorem{remark}[PARA]{Remark}
\theoremstyle{plain}

\theoremstyle{plain}

\newtheorem{definition-proposition}[PARA]{Definition/Proposition}

\newcommand{\cF}{\mathcal{F}}

\newcommand{\cL}{\mathcal{L}}
\newcommand{\cM}{\mathcal{M}}

\newcommand{\cR}{\mathcal{R}}

\newcommand{\R}{{\mathbb{R}}}
\renewcommand{\S}{{\mathbb{S}}}

\newcommand{\Id}{\mathrm{ Id}}
\newcommand{\ev}{\mathrm{ev}}
\newcommand{\Crit}{\mathrm{ Crit}}
\def\NABLA#1{{\mathop{\nabla\kern-.5ex\lower1ex\hbox{$#1$}}}}
\def\Nabla#1{\nabla\kern-.5ex{}_{#1}}
\def\Tabla#1{\Tilde\nabla\kern-.5ex{}_{#1}}

\newcommand{\bM}{\cM^{\dag}}
\newcommand{\bMhyb}[1][\Xi,\Xi']{\cM^{\dag}_{#1}}
\newcommand{\MSLoops}{\cL}%
\let\xto=\xrightarrow%
\newcommand{\fibertimes}[2]{\underset{\mathllap{
      #1},\mathrlap{#2}}{\times}}

\newcommand{\CrocoUp}{\Gamma^{\uparrow}}
\newcommand{\CrocoDo}{\Gamma^{\downarrow}}
\newcommand{\wstar}{%
  \tikz[baseline=(x.base)]{\path(0,0)
    node[scale=-.85,white](x){$\bigstar$}circle(.6ex)
    node[yscale=-1](x){$\star$}circle(.6ex);}
}

\newcommand{\wostar}{\mathbin{%
    \tikz[baseline=(x.base)]{
      \node[white,inner sep=0pt,outer sep=0pt,anchor=center,scale=1.5]
       at (0,0){$\ostar$};
      \node[inner sep=0pt,outer sep=0pt,anchor=center]
      (x) at (0,0){$\ostar$};}} }
\newcommand{\ostar}{\mathbin{{\scriptstyle\textrm{\ding{88}}}}}

\def\crocostar{$\wstar$}

\makeatletter
\long\def\ifnodedefined#1#2#3{%
    \@ifundefined{pgf@sh@ns@#1}{#3}{#2}%
}
\long\def\undefinenode#1{%
    \expandafter\global\expandafter\let\csname pgf@sh@ns@#1\endcsname\@undefined%
}
\makeatother

\tikzset{%
  on each segment/.style={
    decorate,
    decoration={
      post length=1pt,
      pre length=1pt,
      show path construction,
      moveto code={},
      lineto code={\path[#1]
        (\tikzinputsegmentfirst)--(\tikzinputsegmentlast);},
      curveto code={\path[#1]
        (\tikzinputsegmentfirst).. controls
        (\tikzinputsegmentsupporta) and
        (\tikzinputsegmentsupportb)..
        (\tikzinputsegmentlast);},
      closepath code={\path[#1]
        (\tikzinputsegmentfirst)--(\tikzinputsegmentlast);},
    },
  },
  mid arrow/.style={postaction={decorate,decoration={markings,
        mark=at position .5 with {\arrow{#1}}%
      }}},%
  mid arrows/.style 2 args={%
    postaction={decorate,decoration={markings,
        mark=between positions 1/(#1+1) and 1-1/(#1+1) step
        {\pgfdecoratedpathlength/(#1+1)}
        with {\arrow{#2}}%
      }}},%
  ->-/.style={mid arrow={#1}},
  ->-/.default={>[length={2pt 2}]},
}%

\tikzset{
  noEnd/.tip={.[sep=0pt 1.5]
    Butt Cap [length=0pt 1.5,width=0pt 1,sep=0pt 1.5]
    Butt Cap [length=0pt 1.5,width=0pt 1,sep=0pt 2]
    Butt Cap [length=0pt 1. ,width=0pt 1,sep=0pt 2.5]
    Butt Cap [length=0pt  .5,width=0pt 1]
  }
}

\pgfkeys{/jf/myoutline/.cd,
  color/.store in=\jfmyoutlinecolor,
  width/.store in=\jfmyoutlinewidth,
  color=white,
  width=3pt,
}
\newcommand{\tikzoutline}[2][color=white]{%
  \begingroup
  \pgfkeys{/jf/myoutline/.cd,#1}%
  \rlap{
    \textpdfrender{%
      TextRenderingMode=FillStrokeClip,%
      LineWidth=\jfmyoutlinewidth,%
      FillColor=\jfmyoutlinecolor,%
      StrokeColor=\jfmyoutlinecolor,%
      MiterLimit=1%
    }{#2}%
  }
  {#2}%
  \endgroup
}
\pgfdeclarelayer{background}
\pgfdeclarelayer{foreground}
\pgfsetlayers{background,main,foreground}

\newif\ifcrocosavenodes
\newif\ifcrocovertical
\edef\crocosavenodename{CROCO\noexpand\thecroconodenumber}%
\gdef\crocotobeclosedwith{last}%

\tikzset{
  remember points/num/.initial=0,%
  remember points/count/.initial=11,%
  remember points/prefix/.initial=AUTO,%
  remember points/.style={postaction={decoration={
        markings,
        mark=between positions 0 and 1 step
        1/(\pgfkeysvalueof{/tikz/remember points/count}-1) with{
          \pgfmathtruncatemacro{\tmp}{-1+
            \pgfkeysvalueof{/pgf/decoration/mark info/sequence number}}
          \coordinate[name=\pgfkeysvalueof{/tikz/remember points/prefix}\tmp];
        }
      },
      decorate},%
    },%
  }

  \tikzset{
  croco/.style={draw,%
    postaction={decoration=crocofull, decorate},%
  },%
  croco/trajstyle/.style={blue,->-={>[width=1pt 5]},relative,out=10*rand,in=180+10*rand},%
  croco/evalstyle/.style={red},%
  croco/borderstyle/.style={thick,/tikz/croco/trajstyle},
  croco/save node prefix/.code={%
    \edef\crocosavenodename{#1\noexpand\thecroconodenumber}%
  },
  croco/save nodes prefix/.default={CROCO},%
  croco/save nodes/.is if=crocosavenodes,%
  croco/proba/.initial=.75,%
  croco/proba/.default=.75,%
  croco/vertical/.is if=crocovertical,%
  croco/vertical/.default=false,%
  croco/verticality/.initial=0,%
  croco/verticality/.default=0,%
  croco/tobeclosedwith/.initial=last,
  croco/tobecontinuedwith/.initial=first,
  croco/start/.is choice,%
  croco/start/length/.initial=0,%
  croco/start/first/.code=
    \gdef\crocotobeclosedwith{last}%
    \def\crocofirststate{startfirst}
    \pgfkeyssetvalue{/tikz/croco/start}{first}%
    \pgfkeyssetvalue{/tikz/croco/start/length}{.5},%
  croco/start/star/.code=
    \gdef\crocotobeclosedwith{closestar}
    \def\crocofirststate{startstar}
    \pgfkeyssetvalue{/tikz/croco/start}{star}
    \pgfkeyssetvalue{/tikz/croco/start/length}{.5},%
  croco/start/midup/.code=
    \gdef\crocotobeclosedwith{closemidup}
    \def\crocofirststate{startmidup}
    \pgfkeyssetvalue{/tikz/croco/start}{midup}
    \pgfkeyssetvalue{/tikz/croco/start/length}{0},%
  croco/start/continuestar/.code=
    \def\crocofirststate{continuestar}
    \pgfkeyssetvalue{/tikz/croco/start}{conitnue}
    \pgfkeyssetvalue{/tikz/croco/start/length}{.5},%
  croco/start/continuemidup/.code=
    \def\crocofirststate{continuemidup}
    \pgfkeyssetvalue{/tikz/croco/start}{conitnue}
    \pgfkeyssetvalue{/tikz/croco/start/length}{0},%
  croco/start/continue/.code=
    \edef\tmp{{/tikz/croco/start=\crocotobecontinuedwith}}
    \expandafter\pgfkeys\tmp
    \pgfkeyssetvalue{/tikz/croco/start}{continue},%
  croco/start/.default=first,%
  croco/end/.is choice,%
  croco/end/length/.initial=0,%
  croco/end/last/.code=
    \def\crocolaststate{endlast}
    \pgfkeyssetvalue{/tikz/croco/end}{last}
    \pgfkeyssetvalue{/tikz/croco/end/length}{.5},%
  croco/end/star/.code=
    \def\crocolaststate{endstar}
    \pgfkeyssetvalue{/tikz/croco/end}{star}
    \pgfkeyssetvalue{/tikz/croco/end/length}{.5},%
  croco/end/midup/.code=
    \def\crocolaststate{endmidup}
    \pgfkeyssetvalue{/tikz/croco/end}{midup}
    \pgfkeyssetvalue{/tikz/croco/end/length}{0},%
  croco/end/closestar/.code=
    \def\crocolaststate{endclosestar}
    \pgfkeyssetvalue{/tikz/croco/end}{closestar}
    \pgfkeyssetvalue{/tikz/croco/end/length}{.5},%
  croco/end/closemidup/.code=
    \def\crocolaststate{endclose}
    \pgfkeyssetvalue{/tikz/croco/end}{closemidup}
    \pgfkeyssetvalue{/tikz/croco/end/length}{0},%
  croco/end/close/.code=
    \edef\tmp{{/tikz/croco/end=\crocotobeclosedwith}}
    \expandafter\pgfkeys\tmp
    \pgfkeyssetvalue{/tikz/croco/end}{close},%
  croco/end/.default=last,%
}

\newcounter{croconodenumber}
\newdimen\crocosteplength
\newdimen\crocostepheight

\def\crocochoosenextstate{%
  \pgfmathparse{rnd}%
  \ifdim\pgfmathresult pt<\pgfkeysvalueof{/tikz/croco/proba}pt%.66pt%
    \def\croconextstate{upwait}%
  \else
    \def\croconextstate{starupwait}%
  \fi%
}%

\def\crocorot{\ifcrocovertical-\pgfdecoratedangle\else0\fi}
\def\crocorot{
  (Mod(-\pgfdecoratedangle,360)>180?
    Mod(-\pgfdecoratedangle,360)-360:
    Mod(-\pgfdecoratedangle,360)
    )*\pgfkeysvalueof{/tikz/croco/verticality}
}

\pgfdeclaredecoration{crocofull}{initial}%
{
  \state{initial}[
  persistent precomputation={
    \pgfkeys{/tikz/croco/end/.expanded=\pgfkeysvalueof{/tikz/croco/end}}
    \pgfkeys{/tikz/croco/start/.expanded=\pgfkeysvalueof{/tikz/croco/start}}
    \pgfmathparse{ceil(\pgfdecoratedpathlength/\pgfdecorationsegmentlength
      -\pgfkeysvalueof{/tikz/croco/start/length}
      -\pgfkeysvalueof{/tikz/croco/end/length} )}
    \pgfmathparse{\pgfdecoratedpathlength/(\pgfmathresult
      +\pgfkeysvalueof{/tikz/croco/start/length}
      +\pgfkeysvalueof{/tikz/croco/end/length})}
    \setlength\crocosteplength{\pgfmathresult pt}
    \setlength\crocostepheight{\pgfdecorationsegmentamplitude}    
  },
  next state=\crocofirststate,%first,
  width=0pt,
  ]{}
  \state{startmidup}[
  persistent precomputation={
    \crocochoosenextstate
  },
  width=0pt, next state=\croconextstate]%upwait]
  {
    \setcounter{croconodenumber}{-1}
    \coordinate(ctA) at (0,0);
    \coordinate(ctB) at (0,0);
    \ifcrocosavenodes
    \stepcounter{croconodenumber}
    \pgfcoordinate{\crocosavenodename}{\pgfpointanchor{ctA}{center}}
    \fi
    \begin{scope}[rotate=\crocorot]
      \coordinate(ctD) at (0,-2.5\crocostepheight);
      \coordinate(ctC) at (0,-1.5\crocostepheight);
      \coordinate(ctCfirst) at (ctC);
      \coordinate(ctDfirst) at (ctD);
      \coordinate(ceA) at (ctD);
      \coordinate(ceB) at (.25\crocosteplength,-2.5\crocostepheight);
      \coordinate(ceCfirst) at (-.25\crocosteplength,-2.5\crocostepheight);
      \draw[croco/borderstyle](ctA)to(ctC);
      \draw[croco/borderstyle](ctC)to(ctD);
    \end{scope}
  }
  \state{startfirst}[
  width=0pt,
  next state=downwait,
  ]{
    \setcounter{croconodenumber}{-1}
    \begin{pgfonlayer}{foreground}
      \node at (0,0){$\wstar$};
    \end{pgfonlayer}
    \coordinate(ctA) at (0,0);
    \coordinate(ctB) at (0,0);
    \ifcrocosavenodes
    \stepcounter{croconodenumber}
    \pgfcoordinate{\crocosavenodename}{\pgfpointanchor{ctA}{center}}
    \fi
    \begin{scope}[rotate=\crocorot]
      \coordinate(ctD) at (0,0);
      \coordinate(ctC) at (0,0);
      \coordinate(ctDfirst) at (ctD);
      \coordinate(ctCfirst) at (ctC);
      \coordinate(ceA) at (ctA);
      \coordinate(ceB) at (.1\crocosteplength,-\crocostepheight);
      \coordinate(ceCfirst) at (-.1\crocosteplength,-\crocostepheight);
    \end{scope}
  }
  \state{startstar}[
  persistent precomputation={
    \crocochoosenextstate
  },
  width=.5\crocosteplength,
  next state=\croconextstate,%upwait,
  ]{
    \setcounter{croconodenumber}{-1}
    \begin{pgfonlayer}{foreground}
      \node at (0,0){$\wstar$};
    \end{pgfonlayer}
    \coordinate(ctA) at (0,0);
    \coordinate(ctB) at (0,0);
    \ifcrocosavenodes
    \stepcounter{croconodenumber}
    \pgfcoordinate{\crocosavenodename}{\pgfpointanchor{ctA}{center}}
    \fi
    \begin{scope}[rotate=\crocorot]
      \coordinate(ctD) at (.5\crocosteplength,-2.5\crocostepheight);
      \coordinate(ctC) at (.5\crocosteplength,-1.5\crocostepheight);
      \coordinate(ctDfirst) at (ctD);
      \coordinate(ctCfirst) at (ctC);
      \coordinate(ceA) at (ctD);
      \coordinate(ceB) at (.75\crocosteplength,-2.3\crocostepheight);
      \coordinate(ceCfirst) at (-.25\crocosteplength,-2.3\crocostepheight);
      \draw[croco/borderstyle](ctA)to(ctC);
      \draw[croco/borderstyle](ctC)to(ctD);
    \end{scope}
  }
  \state{continuemidup}[
  persistent precomputation={
    \crocochoosenextstate
  },
  width=0pt,
  next state=\croconextstate,
  ]
  {
  }
  \state{continuestar}[
  persistent precomputation={
    \crocochoosenextstate
  },
  width=.5\crocosteplength,
  next state=\croconextstate,%upwait,
  ]
  {
  }
  
  \state{up}[next state=downwait, width=0pt]
  {
    \coordinate(ctA) at (0,0);
    \ifcrocosavenodes
    \stepcounter{croconodenumber}
    \pgfcoordinate{\crocosavenodename}{\pgfpointanchor{ctA}{center}}
    \fi
    \begin{scope}[rotate=\crocorot]
      \coordinate(ctB) at (0,-\crocostepheight);
      \draw[croco/trajstyle](ctA)to(ctB);
      \draw[croco/trajstyle](ctB)to(ctC);
    \end{scope}
  }
  \state{starup}[ next state=downwait,width=0pt]
  {
    \coordinate(ctA) at (0,0);
    \ifcrocosavenodes
    \stepcounter{croconodenumber}
    \pgfcoordinate{\crocosavenodename}{\pgfpointanchor{ctA}{center}}
    \fi
    \begin{pgfonlayer}{foreground}
      \node at (0,0){$\wstar$};
    \end{pgfonlayer}
    \coordinate(ctB) at (0,0);
    \draw[croco/trajstyle](ctB)to(ctC);
  }
  \state{down}[
  persistent precomputation={
    \crocochoosenextstate
  },
  width=0pt,%\stepwidth*.05,
  next state=\croconextstate,%upwait,
  ]
  {
    \begin{scope}[rotate=\crocorot]
      \coordinate(ctC) at (0pt,-1.5\crocostepheight);
      \coordinate(ctD) at (0pt,-2.5\crocostepheight);
      \draw[croco/trajstyle](ctB)to(ctC);
      \draw[croco/trajstyle](ctC)to(ctD);
      \coordinate(ceC) at (-.25\crocosteplength,-2.3\crocostepheight);
      \coordinate(ceD) at (ctD);
      \draw[croco/evalstyle](ceA) ..controls (ceB) and (ceC).. (ceD);
      \coordinate(ceA) at (ceD);
      \coordinate(ceB) at (.25\crocosteplength,-2.3\crocostepheight);
    \end{scope}
  }
  \state{endmidup}[next state=final,width=0pt] {
    \begin{scope}[rotate=\crocorot]
      \coordinate(ctC) at (0,-1.5\crocostepheight);
      \coordinate(ctD) at (0,-2.5\crocostepheight);
      \pgfcoordinate{ctA}{\pgfpointdecoratedpathlast}
      \ifcrocosavenodes
      \stepcounter{croconodenumber}
      \pgfcoordinate{\crocosavenodename}{\pgfpointdecoratedpathlast}
      \fi
      \draw[croco/trajstyle](ctB)to(ctC);
      \draw[croco/borderstyle](ctC)to(ctD);
      \draw[croco/borderstyle](ctA)to(ctC);
      \coordinate(ceC) at (-.25\crocosteplength,-2.3\crocostepheight);
      \coordinate(ceD) at (ctD);
      \draw[croco/evalstyle](ceA) ..controls (ceB) and (ceC).. (ctD);
      \coordinate(ceB) at (.1\crocosteplength,-2.3\crocostepheight);
      \coordinate(ceA) at (ceD);
    \end{scope}
  }
  \state{endlast}[next state=final,width=0pt] {
    \pgfcoordinate{ctA}{\pgfpointdecoratedpathlast}
    \pgfcoordinate{ctB}{\pgfpointdecoratedpathlast}
    \begin{pgfonlayer}{foreground}
      \node [at=(ctA)] {$\wstar$};
    \end{pgfonlayer}
    \ifcrocosavenodes
    \stepcounter{croconodenumber}
    \pgfcoordinate{\crocosavenodename}{\pgfpointdecoratedpathlast}
    \fi
    \draw[croco/trajstyle](ctB)to(ctC);
    \pgfcoordinate{ctC}{\pgfpointdecoratedpathlast}
    \pgfcoordinate{ctD}{\pgfpointdecoratedpathlast}
    \begin{scope}[rotate=\crocorot]
      \coordinate(ceC) at (-.1\crocosteplength,-\crocostepheight);
      \coordinate(ceD) at (ctD);
      \draw[croco/evalstyle](ceA) ..controls (ceB) and (ceC).. (ceD);
      \coordinate(ceB) at (.1\crocosteplength,-\crocostepheight);
      \coordinate(ceA) at (ceD);
    \end{scope}
  }
  \state{endstar}[next state=final,width=0pt] {
    \pgfcoordinate{ctA}{\pgfpointdecoratedpathlast}
    \pgfcoordinate{ctB}{\pgfpointdecoratedpathlast}
    \begin{pgfonlayer}{foreground}
      \node [at=(ctA)] {$\wstar$};
    \end{pgfonlayer}
    \ifcrocosavenodes
    \stepcounter{croconodenumber}
    \pgfcoordinate{\crocosavenodename}{\pgfpointdecoratedpathlast}
    \fi
    \draw[croco/trajstyle](ctB)to(ctC);
    \begin{scope}[rotate=\crocorot]
      \coordinate(ctC) at (.5\crocosteplength,-1.5\crocostepheight);
      \coordinate(ctD) at (.5\crocosteplength,-2.5\crocostepheight);
      \draw[croco/borderstyle](ctA)to(ctC);
      \draw[croco/borderstyle](ctC)to(ctD);
      \coordinate(ceC) at (.25\crocosteplength,-2.3\crocostepheight);
      \coordinate(ceD) at (ctD);
      \draw[croco/evalstyle](ceA) ..controls (ceB) and (ceC).. (ceD);
      \coordinate(ceB) at (.75\crocosteplength,-2.3\crocostepheight);
      \coordinate(ceA) at (ceD);
    \end{scope}
  }
  \state{endclose}[
  next state=final,
  width=0pt,
  ] {
    \begin{scope}[rotate=\crocorot]
      \coordinate(ctC) at (ctCfirst);
      \coordinate(ctD) at (ctDfirst);
      \pgfcoordinate{ctA}{\pgfpointdecoratedpathlast}
      \draw[croco/trajstyle](ctB)to(ctC);
      \coordinate(ceC) at (ceCfirst);
      \coordinate(ceD) at (ctDfirst);
      \draw[croco/evalstyle](ceA) ..controls (ceB) and (ceC).. (ctD);
      \coordinate(ceB) at (.1\crocosteplength,-2.3\crocostepheight);
      \coordinate(ceA) at (ctD);
    \end{scope}
  }
  \state{endclosestar}[
  next state=final,
  width=0pt,
  ] {
    \begin{scope}[rotate=\crocorot]
      \pgfcoordinate{ctA}{\pgfpointdecoratedpathlast}
      \coordinate(ctB) at (ctA);
      \draw[croco/trajstyle](ctB)to(ctC);
      \coordinate(ceC) at (ceCfirst);
      \coordinate(ctD) at (ctDfirst);
      \draw[croco/evalstyle](ceA) ..controls (ceB) and (ceC).. (ctD);
      \coordinate(ceB) at (.75\crocosteplength,-2.3\crocostepheight);
      \coordinate(ceA) at (ctD);
    \end{scope}
  }
  \state{downwait}[
  repeat state=9,
  switch if less than=.051\crocosteplength to \crocolaststate,
  width=.05\crocosteplength,
  next state=down,
  ]{}
  
  \state{upwait}[
  repeat state=9,
  switch if less than=.051\crocosteplength to \crocolaststate,
  width=.05\crocosteplength,
  next state=up,
  ]{
  }
  \state{starupwait}[
  repeat state=9,
  switch if less than=.051\crocosteplength to \crocolaststate,
  next state=starup,
  width=.05\crocosteplength,
  ]{}
  \state{final}{
    \pgfpathmoveto{\pgfpointdecoratedpathlast}}
  
}

    \newcommand\myto[2]{to[style/.expand once={#2}]
      coordinate[pos= 0/10](#10)%
      coordinate[pos= 1/10](#11)%
      coordinate[pos= 2/10](#12)%
      coordinate[pos= 3/10](#13)%
      coordinate[pos= 4/10](#14)%
      coordinate[pos= 5/10](#15)%
      coordinate[pos= 6/10](#16)%
      coordinate[pos= 7/10](#17)%
      coordinate[pos= 8/10](#18)%
      coordinate[pos= 9/10](#19)%
      coordinate[pos=10/10](#110)%
    }
    \tikzset{
      trajectory/.style={->-},
      othertraj/.style={densely dotted},
      eval/.style={blue}
    }
    \newcommand{\upperstep}[1]{
      \coordinate(g0) at (#1,3);
      \coordinate(g1) at (#1+2,3);
      \coordinate(y0) at (#1,2);
      \coordinate(y1) at (#1+2,2);
      \coordinate(x0) at (#1+1,1);

      \draw[trajectory](g1) to (y1) node{\tiny$\bullet$};
      \draw[trajectory](y1)to(x0);
      \path(g0)\myto{A}{out=0,in=180} (g1);
      \path(y0)\myto{B}{} (y1);
      \foreach \i in {2,4,6,8} {
        \draw[othertraj] plot[smooth] coordinates {(A\i)(B\i)(x0)};
      }
    }
    \newcommand{\lowerstep}[1]{
      \coordinate(y0) at (#1+1,2);
      \coordinate(x0) at (#1  ,1);
      \coordinate(x1) at (#1+2,1);
      \coordinate(a0) at (#1  ,0);
      \coordinate(a1) at (#1+2,0);

      \draw[trajectory](y0) to (x1) node{\tiny$\bullet$};
      \draw[trajectory](x1)to(a1);
      \path(x0)\myto{A}{}(x1);
      \draw[eval](a0)\myto{B}{}(a1);
      \foreach \i in {2,4,6,8} {
        \draw[othertraj] plot[smooth] coordinates {(y0)(A\i)(B\i)};
      }
    }
    \newcommand{\upperstarstepleft}[1]{
      \coordinate(g0) at (#1  ,3);
      \coordinate(s0) at (#1+2,3);
      \coordinate(y0) at (#1  ,2);
      \coordinate(x0) at (#1+1,1);

      \begin{pgfonlayer}{foreground}
        \draw[trajectory](s0) node{\crocostar} to (x0);        
      \end{pgfonlayer}
      \path(g0)\myto{A}{out=0,in=180}(s0);
      \path(y0)\myto{B}{}($(s0)!.5!(x0)$);
      \foreach \i in {2,4,6,8} {
        \draw[othertraj] plot[smooth] coordinates {(A\i)(B\i)(x0)};
      }
    }
    \newcommand{\upperstarstepright}[1]{
      \coordinate(g0) at (#1+2,3);
      \coordinate(s0) at (#1  ,3);
      \coordinate(y1) at (#1+2,2);
      \coordinate(x0) at (#1+1,1);

      \draw[trajectory](g0) to (y1)node{\tiny$\bullet$};
      \draw[trajectory](y1) to (x0);
      \path(s0)\myto{A}{out=0,in=180}(g0);
      \path($(s0)!.5!(x0)$)\myto{B}{}(y1);
      \foreach \i in {2,4,6,8} {
        \draw[othertraj] plot[smooth] coordinates {(A\i)(B\i)(x0)};
      }
    }
    \newcommand{\lowerstarstep}[1]{
      \coordinate(s0) at (#1+1,3);
      \coordinate(x0) at (#1  ,1);
      \coordinate(x1) at (#1+2,1);
      \coordinate(a0) at (#1  ,0);
      \coordinate(a1) at (#1+2,0);

      \draw[trajectory](s0) to (x1) node{\tiny$\bullet$};
      \draw[trajectory](x1)to(a1);
      \path(x0)\myto{A}{}(x1);
      \draw[eval](a0)\myto{B}{}(a1);
      \foreach \i in {2,4,6,8} {
        \draw[othertraj] plot[smooth] coordinates {(s0)(A\i)(B\i)};
      }
    }
    \newcommand{\firststep}[1]{
      \coordinate(s0) at (#1  ,3);
      \coordinate(x0) at (#1+1,1);
      \coordinate(a0) at (#1+1,0);
      \draw[trajectory](s0) to (x0);
      \draw[trajectory](x0) to (a0);
      \draw[eval](s0)\myto{A}{out=-90,in=-90}(a0);
      \foreach \i in {2,4,6,8} {
        \draw[othertraj] (s0)to[out=0,in=180-45](A\i);
      }
      \draw (x0) node{\tiny$\bullet$};
      \begin{pgfonlayer}{foreground}
        \draw(s0) node{\crocostar};
      \end{pgfonlayer}
    }
    \newcommand{\laststep}[1]{
      \coordinate(s0) at (#1+1,3);
      \coordinate(x0) at (#1  ,1);
      \coordinate(a0) at (#1  ,0);
      \draw[eval](s0)\myto{A}{out=90,in=90}(a0);
      \foreach \i in {2,4,6,8} {
        \draw[othertraj] (s0)to[out=0,in=180+45](A\i);
      }
      \draw (s0) node{\tiny$\bullet$};
    }

\newcommand{\footremember}[2]{
\footnote{#2}
\newcounter{#1}
\setcounter{#1}{\value{footnote}}
}
\newcommand{\footrecall}[1]{
\footnotemark[\value{#1}]
} 
\title{Fundamental group in stable Morse theory.}
\author{%
  Jean-Fran\c{c}ois Barraud\footremember{Toulouse}{IMT, Universit\'e de Toulouse}
  \and
  Florian Bertuol \footrecall{Toulouse}
}
\date{\today}

\ifCompileAll
\else

\fi

\begin{document}

%%%\begin{tikzpicture}
%%%  \pgfsetlayers{background,main,foreground}
%%%  \path[save path=\toto](0,0)to[relative,out= 45,in=180+45](4,0);
%%%  \path[save path=\titi](0,0)to[relative,out=-45,in=180-45](4,0);
%%%  \makeatletter
%%%  \g@addto@macro\toto\titi
%%%  \pgfsyssoftpath@setcurrentpath{\toto}
%%%  \draw[];
%%%  \makeatother
%%%
%%%  \begin{scope}[
%%%    decoration={segment length=50pt,amplitude=30pt},
%%%    croco/trajstyle/.style={
%%%      ->-={>[width=1pt 5]},
%%%      blue,relative,out=10*rand,in=180+10*rand},
%%%    croco/proba=.75,
%%%    croco/start=continue,
%%%    croco/end=close,
%%%    croco/save nodes=true,
%%%    croco/save node prefix=BB,
%%%    croco/vertical=true,
%%%    ]
%%%
%%%    \draw[
%%%    croco,
%%%    croco/start=first,
%%%    croco/end=midup,
%%%    ]
%%%    % (-3,0)to[out=90,in=90](3,0);
%%%    (-5,2)--(0,0);
%%%    
%%%    \draw[
%%%    croco,
%%%    croco/start=continue,
%%%    croco/end=last,
%%%    ]
%%%    % (3,0)to[out=-90,in=-90](-3,0);
%%%    (0,0)--(5,2);
%%%    
%%%
%%%  \end{scope}
%%%
%%%
%%%\end{tikzpicture}
%%%
%%%\end{document}

%%%%%%%%%%%%%%%%%%%%%%%% 
%\SkipFromHere
%%%%%%%%%%%%%%%%%%%%%%%%

\maketitle

%%%%%%%%%%%%%%%%%%%%%%%%%%%%%%%%%%%%%%%%%%%%%%
%%%%%%%%%%%%%%%%%%%%%%%%%%%%%%%%%%%%%%%%%%%%%%
%%%%%%%%%%%%%%%% Abstract %%%%%%%%%%%%%%%%%%%%
%%%%%%%%%%%%%%%%%%%%%%%%%%%%%%%%%%%%%%%%%%%%%%
%%%%%%%%%%%%%%%%%%%%%%%%%%%%%%%%%%%%%%%%%%%%%%

\begin{abstract}
  Morse theory relates algebraic topology invariants and the dynamics of
  the gradient flow of a Morse function, allowing to derive information
  about one out of the other. In the case of the homology, the
  construction extends to much more general settings, and in particular
  to the infinite dimensional setting of the celebrated Floer homology in
  symplectic geometry. The case of the fundamental group is quiet
  different however, and the object of this paper is to provide a
  dynamical description of the fundamental group in the stable Morse
  setting, which can be thought of as an intermediate case between the
  Morse and the Floer settings. 
\end{abstract}

%The object of this paper is to explain how to recover the fundamental
%group of a manifold out of the dynamics of the gradient of a stable
%Morse function on this manifold, i.e. of a Morse function on
%  $M\times\R^{N}$ that is quadratic at infinity.

\tableofcontents

\section{Introduction}

Morse theory owes a large part of its beauty to its ability to describe
abstract algebraic invariants in terms of dynamical properties of
gradient flows, allowing to use one to derive information on the other.
The case of the homology is of particular interest, as it gave birth to
the celebrated Floer homology, which is based on Morse theoretic ideas
applied to an infinite dimensional setting, and has an ever-growing
number of deep applications in the field of symplectic and contact
geometry.

An intermediate situation between Morse and Floer theory is the case of
stabilized Morse theory. In this paper, a stable Morse function on a
closed manifold $M$ will be a Morse function on $M\times\R^{N}$ that is
``quadratic at infinity'', i.e. coincide with a quadratic form outside of
a compact set.

Stable Morse functions are particularly relevant to symplectic
geometry, first of all as generating functions of some Lagrangians
submanifolds,  but also because stable Morse theory can be thought of
as a simplified finite dimensional model for Floer theory.

\medskip

In the present paper, the algebraic invariant under interest is the
fundamental group. The Morse theoretic description of this group is well
known, and provides a presentation of the group in terms of index $0$,
$1$ and $2$ critical points.

This construction was extended to the Floer setting by the first author,
at least regarding generators, in \cite{FloerPi1}. However, a satisfying
intrinsic description of the relations was still missing.

The object of this paper is to provide a full description of the
fundamental group, i.e. for both generators and relations, in the stable
Morse setting. This can be seen as a first step toward the Floer setting,
as all what is done in the current paper is expected to hold in the
Floer setting, the main difference lying in the discussion of transversality
issues.

\medskip

Since $M$ is a deformation retract of $M\times \R^{N}$, one may expect
little differences between the stable and non stable Morse description of
the fundamental group. However, the global dynamical behaviour of the
gradient flow of Morse and stabilized Morse functions differ drastically.
In particular, a well known result of M. Damian in \cite{Mihai} shows
that, in contrast to genuine Morse function, stable Morse functions may
have strictly less critical points than the minimal number of generators
of the fundamental group: this ruins any hope to have a construction of
the fundamental group in the stable setting directly derived from the
non stable setting (see section \ref{sec:MorseVersusStableMorse} for more
explicit examples of dynamical differences between the two situations
relevant to the construction of the fundamental group).

\bigskip

A (stable) Morse data on a closed manifold $M$ consists of a Morse
function $f$ on $M\times\R^{N_{+}}\times\R^{N_{-}}$ for some integers
$N_{+},N_{-}$ such that, for $(p,x_{+},x_{-})$ out of a compact set in
$M\times\R^{N_{+}}\times\R^{N_{-}}$, we have:
$$
f(p,x_{+},x_{-})=\Vert x_{+} \Vert^{2}-\Vert x_{-} \Vert^{2}
$$
where $\Vert \cdot \Vert$ is the canonical euclidean metric on
$\R^{N_{\pm}}$, together with a pseudo gradient for this function.

%Fix a stabilization $M\times\R^{N_{+}}\times\R^{N_{-}}$ of a closed
%manifold $M$. A stable Morse function on $M$ is a Morse function $f$
%on $M\times\R^{N_{+}}\times\R^{N_{-}}$ such that, for $(p,x_{+},x_{-})$
%out of a compact set in $M\times\R^{N_{+}}\times\R^{N_{-}}$, we have:
%$$
%f(p,x_{+},x_{-})=\Vert x_{+} \Vert^{2}-\Vert x_{-} \Vert^{2}
%$$
%where $\Vert \cdot \Vert$ is the canonical euclidean mertic on
%$\R^{N_{\pm}}$. Throughout this paper, the index of a critical point $x$
%of $f$ will stand for its Morse index shifted by $N_{-}$:
%\begin{equation}
%  \label{eq:index}
%  |x|=\ind(x)+N_{-}.
%\end{equation}
%
%A stable Morse data on $M$ consists in a stable
%Morse function $f$, and pseudo gradient $X$ for $f$ (see section
%\ref{sec:setting} for more precise definitions). 

Given a (stable) Morse data on a based manifold $(M,\star)$, the dynamics
of the flow of the pseudo-gradient allows for the definition of a
collection of moduli spaces, in particular, ``augmentation like''
trajectories as well as ``bouncing'' trajectories (see section
\ref{sec:moduli spaces} for more details), out of which the fundamental
group can be recovered:
\begin{theorem}\label{thm:Main}
  For a generic choice of (stable) Morse data, the relevant moduli spaces
  can be used to define a group $\cL$ of ``Morse loops'', together with a
  normal subgroup $\cR$ generated by preferred ``Morse relations'', such
  that there is a canonical ``evaluation map'' that induces an isomorphism
  $$
  \begin{tikzcd}
    \cL/\cR  \ar[r,"\sim","\ev_{*}"'] & \pi_{1}(M,\star).
  \end{tikzcd}
  $$
\end{theorem}

\begin{remark}
  More precisely, the connected components of $1$ dimensional moduli
  spaces of augmentation like trajectories can be seen as ``steps'', and
  loops are sequences of consecutive steps. Although the group of Morse
  loops is not the free group generated by the steps because of the
  consecutivity condition, it is not hard to see \cite{FloerPi1} that
  there are enough of them to generate the fundamental group.

  In particular, the number of components of the moduli spaces of
  augmentation like trajectories define a notion of multiplicity
  $\nu(y)$ for the index $1$ critical points of $f$, and $\nu(\star)$ for
  the base point, such that%
  \begin{equation}
    \label{eq:CountWithMultiplicity}%
    \nu(\star)+\sum_{|y|=1}\nu(y)\geq \mu(\pi_{1}(M,\star)),  
  \end{equation}
  where $\mu(\pi_{1}(M,\star))$ is the minimal number of generators of
  the fundamental group. %(see \cite{FloerPi1}). Déjà cité au-dessus ?

  Notice that in the genuine (non-stable) Morse case, we always have
  $\nu(y)=1$ and $\nu(\star)=0$, so that the above inequality is a
  generalisation to the stable case of the well known inequality
  $$
  \sharp \Crit_{1}(f)\geq \mu(\pi_{1}(M,\star))
  $$
  in the genuine Morse setting. The equality
  \eqref{eq:CountWithMultiplicity} comes as a counterpart to the
  result of M.Damian: the function $f$ may have less critical points than
  the minimal number of generators for the fundamental group, but this
  has to be balanced by non trivial multiplicities, which means more
  complicated moduli spaces than in the Morse case.

%  In other words, since components of the moduli
%  spaces of augmentation like trajectories can be thought of as a
%  stable notion of unstable manifolds, this inequality can be
%  interpreted in the following way: \emph{ if there are not enough
%    critical points, some of them (or the base point itself) have to take
%    on several unstable manifolds by themselves}. 
\end{remark}

\begin{remark}
  In contrast to (non stable) Morse theory, the present description
  of the relations does not lead to numerical constraints. Indeed, the
  family of generators for the relations mentioned in the statement may
  be infinite, and although we a posteriori know that only finitely of
  them are required, it is not clear from our description how to proceed
  to such a selection.
\end{remark}

\begin{remark}
  One very powerful feature of Morse and Floer homology, is its
  filtration according to the action. Unfortunately, the present
  description of the relation relies on ``bouncing'' trajectories, which
  ruins the action control. However, the example discussed in
  \ref{sec:stabS1} shows that some some kind of generalized trajectory,
  that allow to somehow go up the flow, has to be involved anyway.
\end{remark}

Finally, we prove that this construction is functorial, in the following
sense:
\begin{theorem}
  Consider two based manifolds $(M,\star)$ and $(M',\star')$, and a
  continuous map $\phi:(M,\star)\to (M',\star')$. To a generic choice of
  stable Morse data $\Xi$ and $\Xi'$ on $M$ and $M'$ and of a smooth
  perturbation $\tilde{\phi}$ of $\phi$, is associated a group morphism
  $\pi_{1}(M,\star;\Xi)\xto{\phi_{*}}\pi_{1}(M',\star';\Xi')$, defined
  out ``hybrid'' moduli spaces and independent of the perturbation $\tilde{\phi}$, with
  the following properties:
  \begin{enumerate}
  \item
    if $\phi$ and $\psi$ are homotopic as based maps, then
    $\phi_{*}=\psi_{*}$.
  \item
    $\psi_{*}\circ\phi_{*}=(\psi\circ\phi)_{*}$,
  \item
    $\Id_{*}=\Id$,
  \end{enumerate}
\end{theorem}

\bigskip

The organisation of the paper is as follows: in the first part, we
recall basic definition of stable Morse setting, and discuss a few
situations where differences in the dynamical behaviour of the gradient
flow show up, that are particularly relevant to the construction of the
fundamental group. The second part is a review of all the moduli spaces
of trajectories involved in the construction. The generators of the group
are defined in the third part, as well as the ``crocodile walk'', which
is the main technical ingredient of the construction. The fourth part is
dedicated to the definition of the relations, and to the proof of theorem
\ref{thm:Main}. Finally, the functoriality properties are discussed in
the last part.

\section{Settings}

\subsection{Stable Morse setting}\label{sec:MorseVersusStableMorse}
Let $M$ be a closed $n$ dimensional manifold. A stabilized Morse function
on $M$ is a Morse function on $M\times\R^{N+}\times\R^{N-}$, that is
quadratic at infinity, i.e. such that, out of a compact set,
$f(p,x_{+},x_{-}) = q(x_{+},x_{-}) = \Vert x_{+} \Vert^{2}-\Vert x_{-}
\Vert^{2}$ where $\Vert \cdot \Vert$ stands for the standard Euclidean
metric on $\R^{N_{\pm}}$.

A vector field $X$ on $M\times\R^{N+}\times\R^{N-}$ is a pseudo gradient
for a stable Morse function $f:M\times\R^{N+}\times\R^{N-}\to\R$ if
\begin{itemize}
\item
  $df(X)> 0$ away from the critical points of $f$,
\item
  each critical point is the center of a Morse chart in which $X$ is the
  gradient of the standard Morse model for the standard metric,
\item
  outside of a compact set, $X$ is the gradient of $q$ for a product metric
  that is standard in the fiber $\R^{N+}\times\R^{N-}$.
\end{itemize}

\begin{definition}
A choice $\Xi=(f,X)$ of a stable Morse function and a pseudo gradient
for this function is called a Morse auxiliary data.  
\end{definition}

Let $\cF$ be the space of all possible choices of Morse auxiliary data
(for a fixed $N_{+}$ and $N_{-}$). Suitable regularity conditions turn
this space into a Banach manifold, and a condition will be called
generic if it holds for a countable intersection of open dense subsets in
$\cF$.

We also fix a base point $\star$ in $M$, and similarly, a condition on a
couple $(\Xi,\star)$ will be called generic if it holds for a countable
intersection of open dense sets in $\cF\times M$.

\subsection{Morse versus stable Morse settings}

As already mentioned in the introduction, despite the homotopical
equivalence between  $M$ and $M\times\R^{N_{+}}\times\R^{N_{-}}$, the
differences of dynamical behaviour between the stable and non
stable settings ruins any hope of a direct extension of the Morse
description of the fundamental group to the stable Morse setting.

Damian's result \cite{Mihai} about stable Morse functions having
strictly less critical points than the minimal number of generators of
the fundamental group is one instance of this issue, but the dynamical
difference between the two situations is easy to spot in a far less
precise and subtle way: for instance, in the stable setting, almost
all the trajectories run from $+\infty$ to $-\infty$. On the other hand,
a very general strategy to turn a topological object, like a loop in our
case, into a Morse theoretic object, is to push it down by the flow, and
see what it becomes. In the stable setting, this strategy fails
completely, as it simply makes everything disappear at infinity...

\medskip

Notice that this difference of dynamical behaviour is not such a problem
regarding the construction of the Homology, as what is relevant there are
the intersections of stable and unstable manifolds, rather than stable or
unstable manifolds themselves.

Regarding the fundamental group however, the situation is very different.
In the genuine Morse case, the main characters of the construction are
the whole unstable manifolds of the index $1$ critical points in this case:
they define paths in the manifold, and any loop in the manifold is push
down to a sequence of consecutive such paths by the flow.

The key point in our construction is the observation that the stable
counterpart to the usual unstable manifolds
are the moduli spaces $\cM(x,M)$ of ``augmentation like trajectories'',
i.e. of trajectories that start at a critical point $x$ and hit
$M\times\R^{N_{+}}$. They are indeed a stable version of the ``Latour
cells'' (see \cite{Latour}), with the crucial difference that they are
not cells anymore, i.e. that we have no control on their topology at all.

The above difference of behaviour of the flow in the genuine and stable
settings affects the generators. We continue this section with two
examples of situations enlightening differences that are relevant for
the relations. They make no contribution to the final construction and may
be skipped by the reader, but they still illustrate that some aspects of
the construction that may seem artificially far from the original Morse
construction are in fact mandatory.

\paragraph{A stabilized circle with an extra generator}\label{sec:stabS1}

Consider the case where $M=\S^{1}$, and  pick a Morse function $f_{0}$ on
$\S^{1}$ having only one maximum $a$ and one minimum $b$ as critical
points. There are two flow lines from $a$ to $b$: they form a loop that
is the generator of $\pi_{1}(\S^{1})$.

We can now stabilize $f_{0}$ by considering $f:\S^{1}\times\R^{2}\to\R$
defined as
%\footnote{This functions has critical points on
%  $\SS^{1}\times\{0\}$, which definitly not generic~!}
$f(\theta,x_{+},x_{-})=f_{0}(\theta)+x_{+}^{2}-x_{-}^{2}$.

For the standard product metric, nothing changed: we still have exactly
two flow lines from the ``index $1$'' (in a shifted sense) critical point
$(a,0,0)$ to the ``index $0$'' critical point $(b,0,0)$.

We can however modify the metric so that the stable manifold (in the
usual sense) of $b$ now makes a fold as in figure \ref{fig:stabilizedS1}.

We end up with the same critical points, but 4 flow lines,  which give
rise to two generators for the fundamental group. The striking point
here, is that there are no index $2$ critical point, so no obvious way to
kill this extra generator.

This example illustrates that the relations cannot be associated to index
$2$ critical points only, nor can they use the critical points in the
usual way only. The construction proposed in this paper involves
``bouncing'' trajectories: this example illustrates that this is not an
artefact, and that some sort of generalized trajectory has to be involved. 

\begin{figure}
  \centering
  \begin{tikzpicture}[x={(-.5cm,-.3cm)},y={(1cm,-.1cm)},z={(0cm,1cm)}]
    \coordinate(a) at (0,2.25,0);
    \coordinate(b) at (0,-3,0);
    \coordinate(t1)  at (0,0, .5);
    \coordinate(t0)  at (0,0,  0);
    \coordinate(t-1) at (0,0,-.5);
    \coordinate(t-0) at (0,-5.5, 0);
    \coordinate(tt1)  at (0,0, .5);
    \coordinate(tt0)  at (0,0,  0);
    \coordinate(tt-1) at (0,0,-.5);
    \coordinate(tt-0) at (0,5,  0);
    %\draw[thin](0,0,0)--(3,0,0)(0,0,0)--(0,0,3);
    %\draw[thick](0,-5,0)--(0,-4,0)(0,4,0)--(0,5,0);
    \draw[thin](3,0,0)--(3,-4,0)--(-3,-4,0)--(-3,0, 0)
          ..controls +(1,0,0.2) and +(-1,0,0)..
    (t1)  ..controls +(2,0,0) and +(2,0,0)..
    coordinate[pos=.63](fold)coordinate[pos=.12](bfold)
    (t0)  ..controls +(-2,0,0) and +(-2,0,0)..coordinate[pos=.32](foldd)
    (t-1) ..controls +(1,0,0) and +(-1,0,-.2)..
    cycle;
    \draw(fold) to[relative,out=-5, in=180+5](0.5,-1.75,0);
    \draw(foldd)to[relative,out= 0, in=180](bfold);

    \begin{scope}[very thin,>={Stealth[length=2.5pt]}]
      \foreach\a in{0,45,...,359}{
        \draw[-<](b)--+($cos(\a)*(0,.5,0)+sin(\a)*(.5,0,0)$);
        \draw(b)--+($cos(\a)*(0,.6,0)+sin(\a)*(.6,0,0)$);
      }
      \draw[->](b)--+(0,0, 1);
      \draw(b)--+(0,0, 1.1);
      \draw[->](b)--+(0,0,-.75);
      \draw(b)--+(0,0,-.85);
    \end{scope}

    \draw[red,thick,trajectory](t1)  ..controls +(0,-1,0) and +(-1,.25,0).. (b);
    \draw[blue,thick,trajectory](t-1) ..controls +(0,-1,0) and +( 1,.25,0).. (b);
    \draw[red,thick,trajectory](t0)  -- (b);
    \draw[noEnd-,blue,thick,trajectory](t-0) -- (b);

    \draw[fill=white, opacity=.6,postaction={opacity=1,draw}]
    (0,0,1.5)--(0,4,1.5)--(0,4,-1.5)--(0,0,-1.5)--cycle;

    \begin{scope}[very thin,>={Stealth[length=2.5pt]}]
      \foreach\a in{0,45,...,359}{
        \draw[->](a)--+($cos(\a)*(0,.5,0)+sin(\a)*(0,0,.5)$);
        \draw(a)--+($cos(\a)*(0,.6,0)+sin(\a)*(0,0,.6)$);
      }
      \draw[-<](a)--+( 2,0,0);
      \draw(a)--+( 2.2,0,0);
      \draw[-<](a)--+(-2,0,0);
      \draw(a)--+(-2.2,0,0);
    \end{scope}

    \draw[red,thick,trajectory](a) ..controls +( 0,-.5, .25) and +(0, 1,0).. (tt1) ;
    \draw[blue,thick,trajectory](a) ..controls +( 0,-.5,-.25) and +(0, 1,0).. (tt-1);
    \draw[red,thick,trajectory](a) -- (tt0) ;
    \draw[-noEnd,blue,thick,trajectory](a) -- (tt-0);

    \draw[noEnd-noEnd,thick](0,-5,-2)--+(0,10,0);
    \draw[very thin](0,-4,-2)--+(0,0,4)node[below right]{$\R^{-}$};
    \draw[very thin](0,-4,-2)+(-2,0,0)--+(2,0,0)node[below right]{$\R^{+}$};
    \draw(a) node{\tiny$\bullet$} node[above right]{\tikzoutline{$a$}};
    \draw(b) node{\tiny$\bullet$} node[above left]{\tikzoutline{$b$}};
    \draw(0,5,-2) node[below right]{$\S^{1}$};
    % \draw[blue](0,0,0)--(0,0,1);
    \draw(b) node{\tiny$\bullet$} node[above left]{\tikzoutline{$b$}};
    \draw(a)+(0,0,-2) node{\tiny$\bullet$} node[below]{$\pi(a)$};
    \draw(b)+(0,0,-2) node{\tiny$\bullet$} node[below]{$\pi(b)$};

    \begin{scope}[x={(1cm,0cm)},y={(0cm,1cm)},shift={(0,4)}]
      \coordinate(a)at(0, 1);
      \coordinate(b)at(0,-1);
      \draw[->-,thick,blue](a)arc[start angle=90, end angle=-90,radius=1];
      \draw[->-,thick,blue](a)arc[start angle=90, end angle=270,radius=1];
      \draw(a)node{\tiny$\bullet$};
      \draw(b)node{\tiny$\bullet$};
      \draw[->-,thick,red](a)--
      ($(a)+(8pt,0pt)$)arc[start angle=90, end angle=-90,radius=1]--(b);
      \draw[->-,thick,red](a)--
      ($(a)+(16pt,0pt)$)arc[start angle=90, end angle=-90,radius=1]--(b);
      \draw(a)node{\tiny$\bullet$}node[above]{$\pi(a)$};
      \draw(b)node{\tiny$\bullet$}node[below]{$\pi(b)$};
    \end{scope}
  \end{tikzpicture}
  \caption{A stabilized circle with an extra generator,
    but no index $2$ point to kill it.
  }
  \label{fig:stabilizedS1}
\end{figure}
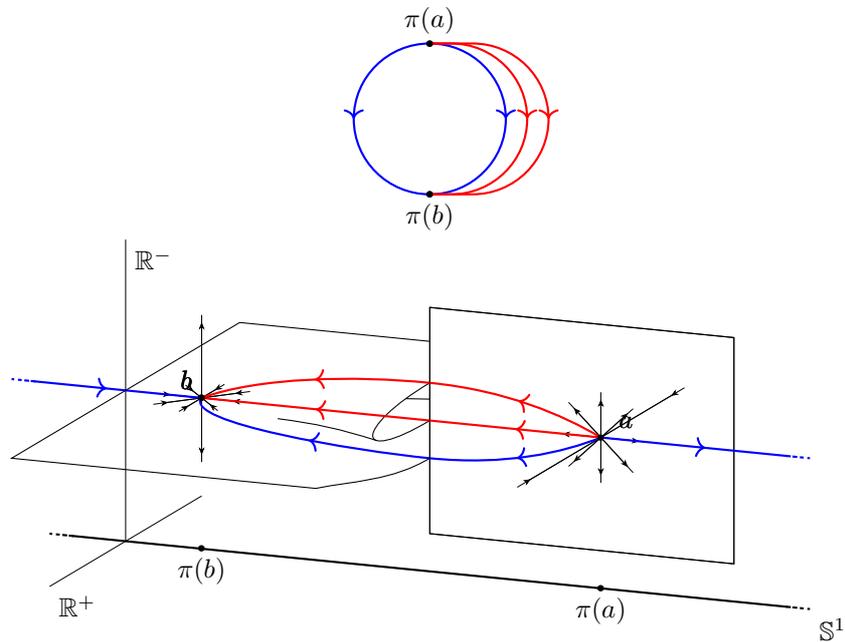

\paragraph{Birth death bifurcations}

Another instructive situation to observe is that of birth/death
bifurcations.

Recall that a generic pair of a Morse function and a Riemannian metric
gives a presentation of the fundamental group. For instance, in the case
where the Morse function has a single minimum, the presentation has one
generator for each index $1$ critical point, and one relation for each
index 2 critical point.

Consider a situation where an index $2$ critical point $x$ and an index
$3$ critical point $y$ are about to cancel each other. This means that
one relation is about to disappear from the presentation.

Of course, since $x$ and $y$ are in a cancelling situation, they are
related by exactly one trajectory, and $x$ and its associated $2$-cell
$W_{x}$ does appear exactly once on the boundary of the $3$-cell
associated to $y$. On the other hand, the boundary of this $3$-cell, in
the Morse setting, is a $2$-sphere, so that the complement of $W_{x}$ in
this $2$ sphere is a disc, and this disc splits as a collection of other
2 cells associated to all the other index $2$ critical points linked to
$y$: the algebraic counterpart of this is that the relation that is
about to disappear from the presentation is in fact a word in other
relations that do persist. This explains the invariance of the group
through this bifurcation.

\medskip

In the stable setting however, genuine unstable manifolds (or more
precisely Latour cells) are useless (as almost all trajectories go
straight to $-\infty$) and should be replaced by moduli spaces of
``augmentation like'' trajectories. Unfortunately, we have no control at
all on the topology these moduli spaces, and the above argument fails
completely.

One way to get rid of this issue is to ``relocate'' the relations at
index $n$ critical points. In the above situation, $x$ itself is
connected by a flow trajectory to an higher index $n$ critical point $z$.
Picking such a trajectory and performing a gluing at $x$ of this
trajectory with all the trajectories in the unstable manifold (more
precisely in the Latour cell) of $x$, the $2$-disc associated to $x$,
becomes a cone rooted at $z$ instead of at $x$ (see figure
\ref{fig:moveroot}), and this cone survives the bifurcation.

%One way to get rid of any requirement on the topology of the unstable
%manifold, is to observe that $x$ is not only related to $y$, but also to
%higher critical points, and at least to one local maximum $z$. The
%$2$-disc associated to $x$, endowed with a trajectory from $z$ to its
%center $x$, can then be seen, after gluing the broken trajectories at
%$x$, as a cone rooted at $z$ instead of $x$ (see figure
%\ref{fig:moveroot}). Since it does not involve neither $x$ nor $y$
%anymore, this cone obviously survives the bifurcation.

Of course this strategy only shifts the issue to birth/death of index $n$
points, but these bifurcations are better controlled (because they are
related to the fundamental class, which has to ``survive'') and can be
sorted out without any requirement on the topology (mainly because the
stable manifolds now have small enough dimension to have
controlled topology).

Our definition of the relations will indeed be associated to index $n$
critical points, which is quiet different from the usual construction,
but this example of birth death bifurcation suggests that some non
trivial modification with respect to the usual construction has to be
done anyway.

\begin{figure}
  \centering
  \begin{tikzpicture}
    \begin{scope}[shift={(-3,0)}]
      \begin{scope}[x={(-.5cm,-.3cm)},y={(1cm,-.1cm)},z={(0cm,1cm)}]
        % \draw(0,0)--(0:1)(0,0)--(90:1);
        \coordinate(z)at(0,0,.5);
        \coordinate(a)at(0,0,2);
        \draw(-60:1)
        \foreach \a in{-30,0,...,120}{
          ..controls +(0,0,.01) and +(\a-90-15:.2).. ($(\a-15:1)+(0,0,.1)$)
          ..controls +(\a+90-15:.2) and +(0,0,.01).. (\a:1)
        };
        \draw(z)..controls +(-60:1) and +(0,0,.1).. (-60:1);
        \draw(z)..controls +(120:1) and +(0,0,.1).. (120:1);
        \foreach \a in{-30,0,...,120}{
          \draw[postaction={decorate,decoration={markings,
              mark=at position .75 with {\arrow{>}}}}]
          (z)..controls +(\a-15:1) and +(0,0,.01).. ($(\a-15:1)+(0,0,.1)$);
        }
        \draw[->-](a)--(z);
        \draw(z)node{\tiny$\bullet$}node[above left]{$x$};
        \draw(a)node{\tiny$\bullet$}node[above]{$z$};     
      \end{scope}
    \end{scope}
    
    \begin{scope}[shift={(0,0)}]
      \begin{scope}[x={(-.5cm,-.3cm)},y={(1cm,-.1cm)},z={(0cm,1cm)}]
        %\draw(0,0)--(0:1)(0,0)--(90:1);
        \coordinate(z)at(0,0,.5);
        \coordinate(a)at(0,0,2);
        \foreach \a in{-30,0,...,120}{
          \coordinate(x\a)at(\a:1);
          \coordinate(y\a)at($(\a-15:1)+(0,0,.1)$);
        }
        \draw(-60:1)
        \foreach \a in{-30,0,...,120}{
          ..controls +(0,0,.01) and +(\a-90-15:.2).. (y\a)
          ..controls +(\a+90-15:.2) and +(0,0,.01).. (x\a)
        };
        \foreach \a in{-30,0,...,120}{
          \draw[postaction={decorate,decoration={markings,
              mark=at position .87 with {\arrow{>}}}}]
          (a)
          ..controls +($(\a:.1)+(0,0,-.3)$) and +($-1*(\a:.3)+(0,0,.2)$)..
          ($(y\a)!.65!(z)+(0,0,.3)$)
          ..controls +($(\a:.2)+(0,0,-.2)$) and +(0,0,.4)..
          (y\a);
        }
        \draw(a)
        ..controls +($(-60:.1)+(0,0,-.3)$) and +($-1*(-60:.3)+(0,0,.2)$)..
        ($(-60:.4)+(0,0,.7)$)
        ..controls +($(-60:.2)+(0,0,-.2)$) and +(0,0,.5)..
        (-60:1);
        \draw(a)
        ..controls +($(120:.1)+(0,0,-.3)$) and +($-1*(120:.3)+(0,0,.2)$)..
        ($(120:.4)+(0,0,.7)$)
        ..controls +($(120:.2)+(0,0,-.2)$) and +(0,0,.6)..
        (120:1);
        \draw(a)node{\tiny$\bullet$}node[above]{$z$};     
      \end{scope}
    \end{scope}
  \end{tikzpicture}
  \caption{Relocating a Morse relation from an index $2$ critical point $x$
    to an index $n$ critical point $z$.}
  \label{fig:moveroot}
\end{figure}
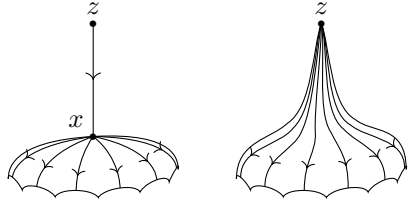

%%%%%%%%%%%%%%%
%\SkipToHere
%%%%%%%%%%%%%%

%%%%%%%%%%%%%%%%%%%%%%%% 
%\SkipFromHere
%%%%%%%%%%%%%%%%%%%%%%%%

%\section{Settings:}
%\begin{itemize}
%\item
%  $M$ is a $n$-dimensional closed manifold,
%\item
%  $f:\bar{M} = M\times\R^{N_{+}}\times\R^{N_{-}}$ is a Morse function
%  that is a quadratic form $q$ at infinity, with signature
%  $(N_{+},N_{-})$.
%\item
%  $g$ is a Riemannian metric on $\bar{M} =  
%  M\times\R^{N_{+}}\times\R^{N_{-}}$ that is standard at infinity.
%\item
%  $X$ is a pseudo gradient for $f$, that is standard in Morse charts near
%  each critical point.
%\item
%  Given a critical point $x$ of $f$, $|x|$ stands for its Morse index
%  shifted down by $N_{-}$:
%  $$
%  |x|=\mu_{\text{Morse}}(x)-N_{-}.
%  $$
%\end{itemize}

\section{Moduli spaces review}\label{sec:moduli spaces}

Before entering the description of the fundamental group, we first review
all the moduli spaces that will be used. On top of the usual moduli
spaces of flow lines, this contains the space of ``augmentation like''
trajectories and ``bouncing'' or ``hybrid'' trajectories.

\subsection{Flow lines, and augmentation like trajectories}
\label{sec:augmentation like}

To reduce notation in the forthcoming definitions, consider the space
$\cF_{I}=\{u:I\to M, u'=-X(u)\}$ of flow trajectories of $-X$ defined on
the interval $I$.

The most basic moduli spaces we are interested in are the following:
\begin{itemize}
\item
  $\cM(x_{-},x_{+})=\left\{u\in\cF_{(-\infty,+\infty)},\ u(\pm\infty)=x_{\pm}\right\}/\R$
\item
  $\cM(x_{-},M)=\left\{u\in\cF_{(-\infty,0]},\ u(-\infty)=x_{-}, u(0)\in
    M\times\R^{N_{+}}\right\}$
\item
  $\cM(M,x_{+})=\left\{u\in\cF_{[0,+\infty)},\ u(0)\in M\times\R^{N_{-}},
    u(+\infty)=x_{-}\right\}$
\item
  $\cM(M,M)=\left\{(T,u),u\in\cF_{[0,T]},\ u(0)\in M\times\R^{N_{-}},
    u(T)\in M\times\R^{N_{-}}\right\}$
\end{itemize}

Trajectories in $\cM(x_{-},M)$ will be called augmentation like
trajectories (and trajectories in $\cM(M,x_{+})$, co-augmentation like
trajectories).

\begin{remark}
  As alternative description, augmentation like trajectories can be seen
  as the ``finite energy'' solutions $u$ of the following ODE:
  $$
  u'(t) =- \chi(t) (X(u(t))- (1-\chi(t))X_{q}
  $$
  where $\chi$ is a smooth function such that $\chi(t)=1$ for $t<0$ and
  $\chi(t)=0$ for $t\geq 0$, $X_{q}$ is the gradient of the quadratic form,
  and the energy of a solution is the $L_{2}$ norm of its derivative
  $\int_{\R}\Vert u'(t) \Vert^{2}dt$.

  From this point of view, they are a strict analog of ``Floer
  cappings'' (i.e. tubes that satisfy the usual Floer equation on one half,
  and are holomorphic on the other, and have finite energy) that are used to define the
  Piunikhin-Salamon-Schwarz morphism for instance, or geometric
  augmentation on some Floer complexes. In this analogy, the quadratic form
  plays the role of the symplectic area, and $(f-q)$ the role of the
  Hamiltonian (see figure \ref{fig:MorseVsFloerAugmentations}).
\end{remark}

These moduli spaces come with maps $\ev_{\pm}$ to $M$ given by the
composition of the evaluation at the beginning or end of the time
interval, and projection on $M$.

Given a point $\star\in M$, we also define
\begin{itemize}
\item
  $\cM(\star,x_{+})=\ev_{-}^{-1}(\star)\subset\cM(M,x_{+})=\{u\in\cM(M,x_{+}),
  u(0)\in\{\star\}\times\R^{N_{-}}\}$
\item
  $\cM(x_{-},\star)=\ev_{+}^{-1}(\star)\subset\cM(x_{-},M)$
\item
  $\cM(\star,M)=\ev_{-}^{-1}(\star)\subset\cM(M,M)$
\item
  $\cM(M,\star)=\ev_{+}^{-1}(\star)\subset\cM(M,M)$
\end{itemize}

It follows from standard transversality arguments that for a generic choice of
auxiliary data $\Xi$ and point $\star$, these moduli spaces are smooth
manifolds.

Notice that although $M\times\R^{N_{+}}\times\R^{N_{-}}$ is not compact,
the behaviour of $f$ at infinity ensures that the flow is complete.
Moreover, our choice of constraints (hitting $M\times\R^{N_{+}}$ at the
positive end of the trajectory or $M\times\R^{N_{-}}$ at the negative
end) ensures that these moduli spaces are compact, up to broken
configurations and zero length trajectories. We will still denote by
$\cM(\cdot,\cdot)$ the compactification of the above moduli spaces. The
expected dimensions for these moduli spaces are the following:
\begin{align*}
  \dim\cM(x_{-},x_{+})&=|x_{-}|-|x_{+}|-1,
  &
    \dim\cM(M,M)&=n+1,
  \\
  \dim\cM(x_{-},M)&=|x_{-}|,
  &
  \dim\cM(M,x_{+})&=n-|x_{+}|,
  \\
  \dim\cM(x_{-},\star)&=|x_{-}|-n,
  &
  \dim\cM(\star,x_{+})&=-|x_{+}|,
  \\
  \dim\cM(M,\star)&=1,
  &
  \dim\cM(\star,M)&=1.
\end{align*}
Moreover, the boundary of these spaces are given by:
\begin{itemize}
\item
  $\partial\cM(x_{-},x_{+})=\bigcup_{y\in\Crit(f)}\cM(x_{-},y)\times\cM(y,x_{+})$,
\item
  $\partial\cM(x_{-},M)=\bigcup_{y\in\Crit(f)}\cM(x_{-},y)\times\cM(y,M)$,
\item
  $\partial\cM(M,x_{+})=\bigcup_{y\in\Crit(f)}\cM(M,y)\times\cM(y,x_{+})$,
%\item
%  $\partial\cM(\star,x_{+})=\bigcup_{y\in\Crit(f)}\cM(\star,y)\times\cM(y,x_{+})$,
%\item
%  $\partial\cM(x_{-},\star)=\bigcup_{y\in\Crit(f)}\cM(x_{-},y)\times\cM(y,\star)$,
%\item
%  $\partial\cM(\star,M)=\bigcup_{y\in\Crit(f)}\cM(\star,y)\times\cM(y,M)$,
%\item
%  $\partial\cM(M,\star)=\bigcup_{y\in\Crit(f)}\cM(M,y)\times\cM(y,\star)$,
\end{itemize}

\begin{figure}
  \centering
  \begin{tikzpicture}
    \begin{scope}[x={(-.5cm,-.3cm)},y={(1cm,-.1cm)},z={(0cm,1cm)}]
      \coordinate(y) at (1,1,1.5);
      \coordinate(p) at (1.25,2,0);
      \coordinate(m) at (0   ,2,0);
      
      \draw[thin](0,0,0)--(2,0,0)--(2,3,0)--(0,3,0)--cycle;
      \draw[thin](0,0,0)--(0,0,2)--(0,3,2)--(0,3,0)--cycle;
      \draw(0,3,0)node[right]{$M$};
      \draw(2,3,0)node[below right]{$\R^{N_{+}}$};
      \draw(0,3,2)node[right]{$\R^{N_{-}}$};
      \draw[blue,thick, trajectory](y) ..controls+(0,.75,0) and +(0,-.25,.5)..
      node[black,pos=.5,above right]{$-\nabla f$}(p);
      \draw(y)node{\tiny$\bullet$}node[above]{$y$};
      \draw[blue,thick,dashed](p)--(m);
    \end{scope}

    \begin{scope}[shift={(6,2)}]
      \begin{scope}%[x={(-.5cm,-.3cm)},y={(1cm,-.1cm)},z={(0cm,1cm)}]
        \draw(0,0) arc(-180:180:1 and .25)coordinate[pos=.15](y);
        \draw[blue, thick](0,0)--(0,-1.5)(2,-1.5)--(2,0);
        \draw[blue, thick, dashed](0,-1.5)--(0,-2) arc(-180:0:1)(2,-2)--(2,-1.5);
        \draw[blue, thick, densely dotted](0,-1.5) arc(-180:0:1 and .25);
        \path (y) node {\tiny$\bullet$}node[above]{$y$};
        \path(1,-1)node{$H$};
        \path(1,-2.25)node{$0$};
      \end{scope}
    \end{scope}
  \end{tikzpicture}
  \caption{Augmentation like trajectories in the stable Morse setting,
    and their Floer analog. } 
  \label{fig:MorseVsFloerAugmentations}
\end{figure}
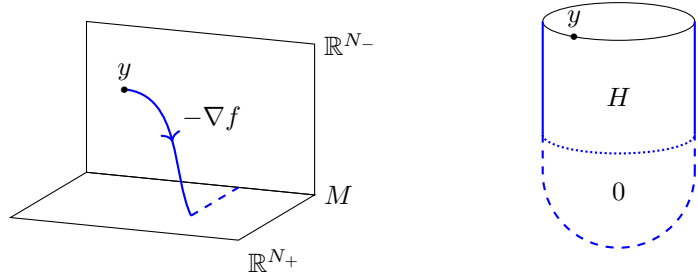

\subsection{Bouncing and hybrid trajectories}\label{def:bouncing trajectory}

We now describe ``bouncing'' trajectories. A bouncing trajectory from a
critical point $x_{-}$ to a critical point $x_{+}$ is an element of:
\begin{align*}
  \bM(x_{-},x_{+})
  &=\cM(x_{-},M)\fibertimes{\ev_{+}}{\ev_{-}}\cM(M,x_{+})\\
  &=\Big\{(u,v)\in\cM(x_{-},M)\times\cM(M,x_{-}),
  \ev_{+}(u)=\ev_{-}(v)\Big\}
\end{align*}

\begin{figure}[ht]
  \centering
  \begin{tikzpicture}[x={(-.5cm,-.3cm)},y={(1cm,-.1cm)},z={(0cm,1cm)}]
    \draw(0,0,0)--(0,3,0)--(0,3,0)--(2,3,0)--(2,3,0)--(2,0,0)--(2,0,0)--(0,0,0);
    \draw(0,0,0)--(0,0,2.5)--(0,3,2.5)--(0,3,0);

    \coordinate (y1) at (.5,1,1.5);
    \coordinate (p1) at (1,1.75,0);
%    \coordinate (y2) at (0,1,1.5);
    \coordinate (p2) at (1,2.5,.75);
    \coordinate (star) at (0,2.5,0);

    \draw[blue,thick,->-] (y1)node[blue]{\tiny$\bullet$}node[above]{$x_{-}$}
    to[out=-120,in= 90]coordinate[pos=.7](t)(p1);
    \draw[dashed](p1)--++(-1,0,0)--++(0,0,1.5)
    coordinate(y2){};
    \draw[thick,->-,red] (y2) to[out=-130,in= 100] (p2);
    \fill[blue](p1) circle (.8pt);
    \fill[red] (y2) circle (.8pt);
    \fill[red] (p2) circle (.8pt);

%    \draw[thin,blue,postaction={draw, thick,
%      decoration={curveto,post=moveto, post length=6mm},decorate}]
%    (y1)to[out=-120,in= 110](p1);
    
    \draw[red](p2)node{\tiny$\bullet$}node[right]{$x_{+}$};
    \draw(0,3,0)node[right]{$M$};
    \draw(0,3,2.5)node[right]{$\R^{N_{-}}$};
    \draw(2,3,0)node[below right]{$\R^{N_{+}}$};

    \begin{scope}[x={(1cm,0cm)},y={(0cm,1cm)},shift={(7,.7)},scale=.5]
      \draw[blue](-1, 3)ellipse(1 and .3);
      \draw[blue](-2, 3)to[out=-90,in=180](0,0)to[out=  0,in=-90](0, 3);
      \path[blue](-2,3)node{\tiny$\bullet$}node[left]{$x_{-}$};
      \begin{scope}[shift={(.25,0)}]
        \draw[red]( 2,-3)to[out= 90,in=  0](0,0)to[out=180,in= 90](0,-3);
        \draw[red]( 0,-3)arc(-180:0:1 and .3);
        \path[red](0,-3)node{\tiny$\bullet$}node[left]{$x_{+}$};
        \draw[red,dashed]( 0,-3)arc(180:0:1 and .3);
      \end{scope}
      \draw[fill](.125,0)circle[radius=.04];
    \end{scope}

  \end{tikzpicture}
  
  \caption{A bouncing trajectory in $\bM(x_{-},x_{+})$, and its Floer analog.
    Hybrid trajectories are bouncing trajectories using different Morse data
    on the different components, which is depicted by a change of color.}
  \label{fig:bouncing once}
\end{figure}
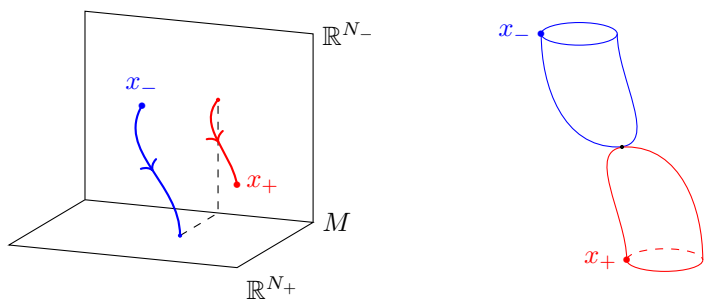

Notice that different auxiliary data can be used before and after the
bounce, in which case we speak of hybrid trajectories.

\begin{remark}
  In the (genuine) Morse setting, bouncing trajectories reduce to
  trajectories with a marked point.

  Hybrid trajectories however are trajectories that flow according to the
  first Morse data down to some arbitrary point, and then flow according
  to the second data. These configurations are typically used to define a
  morphism from the Morse complex associated to the first data, to the
  Morse complex associated to the second one, and derive functoriality
  properties (although an explicit implementation of this construction is
  hard to spot in the literature~; see also the celebrated
  Piunikhin-Salamon-Schwarz morphism \cite{PSS} in the Floer setting, or
  \cite{BDHO} for a discussion of functoriality properties of Morse
  Homology in the context of ``DG'' coefficients).
\end{remark}

Of course the same definition of bouncing trajectories can be given for
any kind of constraints at start or end. In particular, we let:
\begin{align*}
  \bM(x_{-},M)
  &=\cM(x_{-},M)\fibertimes{\ev_{+}}{\ev_{-}}\cM(M,M)\\
  \bM(M,x_{+})
  &=\cM(M,M)\fibertimes{\ev_{+}}{\ev_{-}}\cM(M,x_{+})\\
  \bM(M,M)
  &=\cM(M,M)\fibertimes{\ev_{+}}{\ev_{-}}\cM(M,M)\\
\end{align*}

Again, classical transversality arguments show that for a generic choice
of auxiliary data, these moduli spaces are smooth manifolds with corners.

\begin{definition}\label{def:regular data}
  A choice $(\Xi,\star)$ of a Morse data and a base point $\star$ will be
  said to be regular if all the moduli spaces
  $\cM(x_{-},x_{+}),\cM(x_{-},M),\cM(M,x_{+})$ and $\cM(M,M)$ are cutout
  transversely, and all possible pairs of evaluation maps
  $(\ev_{+},\ev_{-})$ defined on these moduli spaces 
  $$
  \begin{tikzcd}[row sep=1ex,column sep=8ex]
    \cM(x_{-},M)
    \arrow[to=M,start anchor={east},bend left=20]
    & &
    \cM(M,x_{+})
    \arrow[to=M,start anchor={west},bend right=20]\\
    \\
    \cM(\star,M)
    \arrow[to=M,start anchor={east},"\ev_{+}"]
    &|[alias=M]| M &
    \cM(M,\star)
    \arrow[to=M,start anchor={west},"\ev_{-}"']\\
    \\  
    \cM(M,M)
    \arrow[to=M,start anchor={east},bend right=20]
    & &
    \cM(M,M)
    \arrow[to=M,start anchor={west},bend left=20]
  \end{tikzcd},
  $$
  are transverse, i.e. if all the moduli spaces of bouncing trajectories
  $\bM(\cdot,\cdot)$ for all possible constraints at the ends are cut out
  transversely.
\end{definition}
Again, by classical transversality arguments, this condition is a generic
condition.

Notice that a trajectory in $\cM(M,M)$ is parametrized by a bounded
interval in $\R$. This introduces a new boundary component, in addition
to the usual possible break at an intermediate point, as the length of this interval may vanish. The same holds for all the moduli spaces where the trajectories have a component parametrized by a compact interval, in particular 
for moduli spaces of bouncing trajectories with at least one constraint
that is not a critical point, like $\bM(x_{-},M)$, $\bM(M,x_{+})$, or
$\bM(M,M)$.
%
%Observe moreover that for a generic choice of data, there are no critical
%points on $M\times\R^{N_{+}}$ nor $M\times\R^{N_{-}}$, so that this
%vanishing cannot occure one a part constraint to start or end at a
%critical point.

%%Regarding compacity, in addition to the usual possible break at an
%%intermediate point, bouncing trajectories may also undergo vanishing of
%%the length of part of the trajectory above or below the bounce. Notice
%%however that for a generic choice of data, there are no critical points
%%on $M\times\R^{N_{+}}$ nor $M\times\R^{N_{-}}$, so that this vanishing
%%cannot occure one a part constraint to start or end at a critical point.
In particular, we have:
\begin{align*}
  \partial\bM(x_{-},x_{+})
  =&\cM(x_{-},y)\times\bM(y,x_{+})
    \cup
    \bM(x_{-},y)\times\cM(y,x_{+})
  \\
  \partial\bM(x_{-},M)
  =&\cM(x_{-},y)\times\bM(y,M) \cup \bM(x_{-},y)\times\cM(y,M)\cup\\
   &\quad \cup \cM(x_{-},M)',\\
  \partial\bM(M,x_{+})
  =&\cM(M,y)\times\bM(y,x_{+}) \cup \bM(M,y)\times\cM(y,x_{+})\cup\\
   &\quad \cup \cM(M,x_{+})',\\
  \partial\bM(M,M)
  =&\cM(M,y)\times\bM(y,M) \cup \bM(M,y)\times\cM(y,M)\cup\\
   &\quad \cup \cM(M,M)''.
\end{align*}
Here, $\cM(x_{-},M)'$ denotes the inclusion of $\cM(x_{-},M)$ in
$\bM(x_{-},M)$ obtained by adding a zero length trajectory at the end,
$\cM(M,x_{+})'$ by adding it at start, and $\cM(M,M)''$ by adding it
either at end or at start or both.

From now on, we work under the assumption that the choice of Morse data
and base point is regular.

\section{Generators}
In this section, the generators of the stable Morse
fundamental group are defined. This construction is already sketched in
\cite{FloerPi1}, but we recall it there with a little more details.

%\begin{definition}
%  An augmentation like trajectory rooted at a critical point $x\in\Crit(f)$
%  is a path $u:(-\infty,0]\to M\times\R^{N_{+}}\times\R^{N_{-}}$ such
%  that:
%  \begin{enumerate}
%  \item
%    $\forall t\in(-\infty,0], u'(t)=-X(u(t))$
%  \item
%    $\lim_{t\to-\infty}u(t)=x$
%  \item
%    $u(0)\in M\times\R^{N_{+}}$.
%  \end{enumerate}
%\end{definition}

\begin{figure}[ht]
  \centering
  \begin{tikzpicture}[x={(-.5cm,-.3cm)},y={(1cm,-.1cm)},z={(0cm,1cm)}]
    \draw(0,0,0)--(0,4,0)--(0,4,0)--(2,4,0)--(2,4,0)--(2,0,0)--(2,0,0)--(0,0,0);
    \draw(0,0,0)--(0,0,2)--(0,4,2)--(0,4,0);
%    \draw(.5,2,1.5)to[out=-120,in=10](1,1.25,.75)to[out=-100,in=45](1,1,0);
%    \draw(1,1,0)to[out=-5,in=150](1,3,0);
%    \draw(.5,2,1.5)to[out=- 60,in=45](1,2.75,.75)to[out=-60,in=60](1,3,0);
    \coordinate (y) at (.5,2,1.5);
    \coordinate (x1) at (1,1.25,.75);
    \coordinate (x2) at (1,2.75,.75);
    \coordinate (p1) at (1,1,0);
    \coordinate (p2) at (1.5,3,0);

    \fill[fill=white] (y) node{\tiny$\bullet$}node[above]{$y$}
    to[out=-120,in=10,->] (x1)node{\tiny$\bullet$}node[left]{$x_{1}$}
    to[out=-100,in=45] (p1)
    to[out=-  5,in=150](p2)
    to[out= 100,in=-120](x2)node{\tiny$\bullet$}node[right]{$x_{2}$}
    to[out= 140,in=-60](y);

    \draw[blue, thick,->-] (y) to[out=-120,in= 10] (x1);
    \draw[blue, thick,->-](x1) to[out=-100,in= 45] (p1);
    \draw[blue, thick,->-=<](p2) to[out=  100,in=-120] (x2);
    \draw[blue, thick,->-=<](x2) to[out=  140,in=-60] (y);

    \draw(p1)to[out=-  5,in=150]
    coordinate[pos=.2](ww2)
    coordinate[pos=.4](ww4)
    coordinate[pos=.6](ww6)
    coordinate[pos=.8](ww8)
    (p2);
    \draw[densely dotted](y)to[out=-105,in=80](ww2);
    \draw[densely dotted](y)to[out=- 95,in=80](ww4);
    \draw[densely dotted](y)to[out=- 90,in=70](ww6);
    \draw[densely dotted](y)to[out=- 80,in=80](ww8);

    \draw(0,4,0)node[right]{$M$};
    \draw(0,4,2)node[right]{$\R^{N_{-}}$};
    \draw(2,4,0)node[below right]{$\R^{N_{+}}$};
  \end{tikzpicture}
  \caption{A stable Morse step through a critical point $y$.}
  \label{fig:Step-y}
\end{figure}
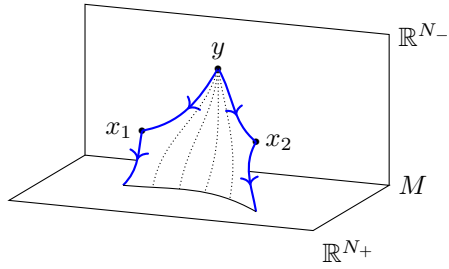

Recall from the previous section that we now work under the assumption
that the Morse data and base point are regular, which implies that the moduli
spaces of augmentation like trajectories are cutout transversely.

In particular, if $y$ is an index $1$ critical point of $f$, $\cM(y,M)$
is $1$-dimensional. When non trivial, the boundary of a component of such
a space consists of two points, which are two broken configurations
$(\beta_{0},\alpha_{0})$ and $(\beta_{1},\alpha_{1})$ where
\begin{itemize}
\item
  $\beta_{i}\in\cM(y,x_{i})$ for some $x_{i}$ such that $|x_{i}|=0$
\item
  $\alpha_{i}\in\cM(x_{i},M)$.  
\end{itemize}

\begin{definition}
  A (stable) Morse step through an index $1$ critical point $y$ is an
  oriented component of the moduli space $\cM(y,M)$ that has non empty
  boundary. A step $\sigma$ will be identified by the ends
  $(\beta_{0},\alpha_{0})$ and $(\beta_{1},\alpha_{1})$ of the component.

  If $\sigma$ is a step, $\bar{\sigma}$ will denote the same step with
  the opposite orientation.

%  Two steps $\sigma=((\beta_{0},\alpha_{0}),(\beta_{1},\alpha_{1}))$ and
%  $\sigma'=((\beta'_{0},\alpha'_{0}),(\beta'_{1},\alpha'_{1}))$ are said to
%  be consecutive $\alpha_{1}=\alpha'_{0}$.
\end{definition}

%Similarly, consider the space $\cMpre(\star,M)$ of couples $(T,u)$, where
%$T\in [0,+\infty)$ and $u:[0,T]\to M\times\R^{N_{+}}\times\R^{N_{-}}$ is
%such that
%\begin{itemize}
%\item 
%  $u'=-X(u)$
%\item
%  $u(0)\in\{\star\}\times\R^{N_{-}}$
%\item
%  $u(T)\in M\times\R^{N_{+}}$.
%\end{itemize}
%
%\begin{remark}
%  Again, this space can also be interpreted as the moduli space of some
%  finite energy solutions $(u,T)$ of the equation:
%  $$
%  u'(t)=-\chi(t)X(u(t))-(1-\chi(t))X_{q}(u(t)),
%  $$
%  where $\chi(t)=1$ if $t\in[0,T]$, and $\chi(t)=0$ otherwise.
%
%  The Floer counterpart of these configurations are the finite energy
%  tubes that satisfy the Floer equation, except near the ends where they
%  are holomorphic.  
%\end{remark}

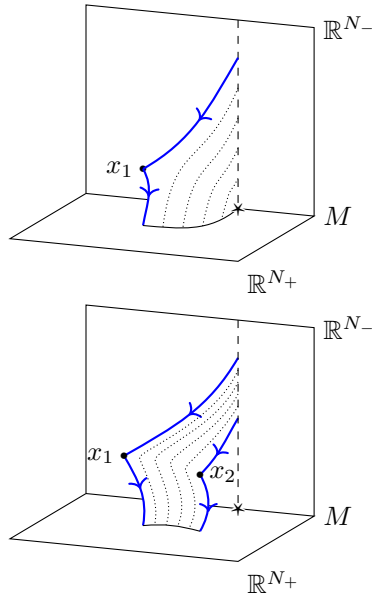
\begin{figure}[ht]
  \centering
  \begin{tikzpicture}[x={(-.5cm,-.3cm)},y={(1cm,-.1cm)},z={(0cm,1cm)}]
    \draw[thin](0,0,0)--(0,3,0)--(0,3,0)--(2,3,0)--(2,3,0)--(2,0,0)--(2,0,0)--(0,0,0);
    \draw[thin](0,0,0)--(0,0,2.5)--(0,3,2.5)--(0,3,0);
%    \draw(.5,2,1.5)to[out=-120,in=10](1,1.25,.75)to[out=-100,in=45](1,1,0);
%    \draw(1,1,0)to[out=-5,in=150](1,3,0);
%    \draw(.5,2,1.5)to[out=- 60,in=45](1,2.75,.75)to[out=-60,in=60](1,3,0);
    \coordinate (y1) at (0,2,2);
    \coordinate (x1) at (1,1.25,.75);
    \coordinate (p1) at (1,1.25,0);
    \coordinate (star) at (0,2,0);
    
    \fill[white] (y1)
    to[out=-120,in= 30] (x1)
    to[out=-60,in= 80] (p1)
    to[out=  20,in=-140] (star);

    \draw[dashed](0,2,2.5)--(star);

    \path(x1)node{\tiny$\bullet$}node[left]{$x_{1}$};

    \draw[blue,thick,->-] (y1) to[out=-120,in= 30] (x1);
    \draw[blue,thick,->-](x1) to[out=-60,in= 80] (p1);

    \draw[remember points, remember points/prefix=ww]
      (p1)to[out=-  5,in=-140](star);
    \path[remember points, remember points/prefix=xx]
      (x1)to(star);
    \path[remember points, remember points/prefix=yy]
      (y1)--(star);
    \foreach \i in {2,4,6,8}{ %
      \draw[densely dotted](yy\i)..controls +(1.1-\i/10,0,-1+\i/10) and
      +(-.3,0,1.25-\i/10) .. (ww\i); 
    }%
    \draw(0,3,0)node[right]{$M$};
    \draw(0,3,2.5)node[right]{$\R^{N_{-}}$};
    \draw(2,3,0)node[below right]{$\R^{N_{+}}$};
    \draw(0,2,0)node{$\wstar$};
  \end{tikzpicture}

  \begin{tikzpicture}[x={(-.5cm,-.3cm)},y={(1cm,-.1cm)},z={(0cm,1cm)},
    remember points/count=11,
    ]
    \draw[thin](0,0,0)--(0,3,0)--(0,3,0)--(2,3,0)--(2,3,0)--(2,0,0)--(2,0,0)--(0,0,0);
    \draw[thin](0,0,0)--(0,0,2.5)--(0,3,2.5)--(0,3,0);

    \coordinate (y1) at (0,2,2);
    \coordinate (y2) at (0,2,1.2);
    \coordinate (x1) at (1,1,.9);
    \coordinate (x2) at (1,2,.75);
    \coordinate (p1) at (1,1.25,0);
    \coordinate (p2) at (1,2,0);
    \coordinate (star) at (0,2,0);

    \fill[white] (y1)
    to[out=-120,in= 30] (x1)
    to[out=-60,in= 80] (p1)--(p2)
    to[out=  60,in=-60] (x2)
    to[out=  45,in=-100] (y2);

    \draw[blue,thick,->-] (y1) to[out=-120,in= 30] (x1);
    \draw[blue,thick,->-](x1) to[out=-60,in= 80] (p1);
    \draw[blue,thick,->-=<](p2) to[out=  60,in=-60] (x2);
    \draw[blue,thick,->-=<](x2) to[out=  45,in=-120] (y2);

    \path(x1)node{\tiny$\bullet$}node[left]{$x_{1}$};
    \path(x2)node{\tiny$\bullet$}node[right]{$x_{2}$};
    \draw[dashed](0,2,2.5)--(star);
    
    \draw[remember points, remember points/prefix=ww]
      (p1) to[out=  5,in=165] (p2);
    \path[remember points, remember points/prefix=xx]
      (x1) to[relative,out=10,in=180-10] (x2);
    \path[remember points, remember points/prefix=yy]
      (y1) to (y2);

    \foreach \i[evaluate=\i as \t using (\i/10*(1-\i/10))] in{2,4,6,8}{
      \draw[densely dotted](yy\i)
      ..controls +(.75,.1,-.5+\i/30) and +(-.03,.0,.03+\t*\t*2).. (xx\i)
      ..controls+(.03,0,-.05)and+(.3,.2+\i/50,.5)..(ww\i);
    }
%    \draw[densely dotted](yy2)to[out=-115,in=40](xx2)to[out=220 ,in=60](ww2);
%    \draw[densely dotted](yy4)to[out=-112,in=45](xx4)to[out=225 ,in=56](ww4); ; 
%    \draw[densely dotted](yy6)to[out=-108,in=50](xx6)to[out=230 ,in=53](ww6); ; 
%    \draw[densely dotted](yy8)to[out=-105,in=55](xx8)to[out=235 ,in=50](ww8);

    \draw(0,3,0)node[right]{$M$};
    \draw(0,3,2.5)node[right]{$\R^{N_{-}}$};
    \draw(2,3,0)node[below right]{$\R^{N_{+}}$};
    \draw(star)node{$\wstar$};
  \end{tikzpicture}
  \caption{The two kinds of stable Morse step through the base point $\star$.}
  \label{fig:Step-star}
\end{figure}

%\begin{proposition}
%  For a generic choice of auxiliray data $\Xi$ and base point $\star$, this
%  space is a $1$ dimensional manifold, and it has a natural
%  compactification $\cM(\star,M)$ for which
%  $$
%  \partial\cM(\star,M)=\{\star\}\bigcup_{x\in\Crit_{0}(f)}\cM(\star,x)\times\cM(x,M)
%  $$
%  Here, $\{\star\}$ stands for the zero-length trajectory at $\star$, and
%  \begin{multline*}
%    \cM(\star,x)=\{u:[0,+\infty)\to M\times\R^{N_{+}}\times\R^{N_{-}}, u'=-X(u),\\
%    u(0)\in\{\star\}\times\R^{N_{-}},\lim_{t\to+\infty}u(t)=x\}.
%  \end{multline*}
%\end{proposition}
%
%From now on, we work under the asumption the the moduli space
%$\cM(\star,M)$ is cutout transversely.

Consider the similar definition where $y$ is replaced by the base point
$\star$:
\begin{definition}
  A stable Morse step through the base point $\star$ is an oriented
  component of the moduli space $\cM(\star,M)$ with non empty boundary.

  Among these components, exactly one (with two possible orientations)
  has the zero-length trajectory at $\star$ on its boundary.

  All the other components have a boundary of the form
  $((\beta_{0},\alpha_{0}), (\beta_{1},\alpha_{1}))$, with
  $\beta_{i}\in\cM(\star,x_{i})$ for some $x_{i}\in\Crit_{0}(f)$, and
  $\alpha_{i}\in\cM(x_{i},M)$.
\end{definition}

%Notice again that the evaluation at the endpoint $T$ of the domain
%$[0,T]$ of a couple $(T,u)\in\cM(\star,M)$ defines a continuous map
%$$
%\ev_{+}:
%\begin{tikzcd}[row sep=1em]
%  \cM(\star,M)\arrow[r]&M\\
%  (T,u)\arrow[r,|->]&\pi(u(T))
%\end{tikzcd}.
%$$

From the notion of Morse step, we derive that of Morse loop:
\begin{definition}
  Two Morse steps (through a critical point or the base point) are
  said to be consecutive if the end $(\beta,\alpha)$ of the first one and
  the start $(\beta',\alpha')$  of the second one are such that
  $\alpha=\alpha'$.

  A Morse loop is a sequence of consecutive Morse steps starting and
  ending at $\star$ (the zero-length trajectory at $\star$).

  Finally, let $\MSLoops(\Xi,\star)$ be the group of Morse loops, for the
  concatenation, modulo the relations 
  $\sigma\bar\sigma=\bar\sigma\sigma=1$.
\end{definition}

\begin{remark}
  We will also speak of stable Morse free loops: these are loops of
  consecutive steps, but the constraint of starting from $\star$ is
  relaxed.
\end{remark}

Consider a Morse step $\sigma$. Since $\sigma$ is an oriented compact
interval, one can pick a parametrization of $\sigma$ by $[0,1]$ and
consider the evaluation map
$$
\ev(\sigma):
\begin{tikzcd}[row sep=1ex]
  [0,1]\arrow[r]& M\\
  \lambda\arrow[r,|->]&\ev_{+} (u_{\lambda})
\end{tikzcd}
$$
(where $\pi:M\times\R^{N_{+}}\to M$ is the first projection). This
defines a path in $M$. It does depend on the choice of parametrization
of $\sigma$, but its homotopy class with fixed ends does not, and we
obtain a well defined map:
\begin{equation}
  \label{eq:evaluationofloops}
  \begin{tikzcd}[row sep = 1ex]
    \MSLoops(\Xi,\star)\arrow[r,"\ev"]& \pi_{1}(M,\star)\\
    (\sigma_{1},\dots,\sigma_{N})\arrow[r,|->]&\ev(\sigma_{1})\cdots\ev(\sigma_{N})
  \end{tikzcd}.
\end{equation}

\begin{theorem}\label{thm:surjectivity}
  The above evaluation map $\ev:\MSLoops(\Xi,\star)\to\pi_{1}(M,\star)$ is surjective.
\end{theorem}

This will be proved in the next section, as a consequence of proposition
\ref{prop:croco homotopy}.

\section{Crocodile Walks}\label{sec:crocodilewalks}

In this section, we recall briefly the crocodile walk procedure from
\cite{FloerPi1}.

Consider a based loop $\gamma:[0,1]\to M$ in $M$, and suppose it is
transversal to all the evaluation maps
$$
\cM(M,x_{+})\xto{\ev_{-}} M
$$
for all the critical points $x_{+}$, and to $\cM(M,M)\xto{\ev_{-}}M$.

Consider then the moduli space
\begin{align*}
  \cM(\gamma,M)
  &=[0,1]\fibertimes{\gamma}{\ev_{-}}\cM(M,M)\\
  &=\{(\tau,u)\in[0,1]\times\cM(M,M),\gamma(\tau)=\ev_{-}(u)\}
\end{align*}
i.e. the space of arcs of flow lines running from some point in
$\gamma\times\R^{N_{-}}$ to some point in $M\times\R^{N_{+}}$. To be more
precise, the configuration also remembers the time parameter along
$\gamma$ of the starting point on $\gamma$, which makes a difference when
$\gamma$ is not injective, and in particular for $\star$, that can be
seen as $\gamma(0)$ or $\gamma(1)$.

This space is a smooth $2$-dimensional manifold with boundary and
corners. Its boundary is hence a collection of topological loops, and
consists of trajectories from $\gamma$ to $M$ that are broken once or
twice, or have zero length.

The corners of $\cM(\gamma,M)$ is the finite set
\begin{multline}
  B
  =
  \bigcup_{y,x\in\Crit(f)}\cM(\gamma,y)\times\cM(y,x)\times\cM(x,M)\\
  \bigcup_{x\in\Crit(f)}\cM(\star,x)\times\cM(x,M)
  \ \bigcup\  \{\star\}.
\end{multline}

It consists of:
\begin{itemize}
\item
  trajectories that are broken twice, in which case we call the first
  break along the trajectory the upper degeneracy, and the second the
  lower one,
\item
  trajectories that start from one end of $\gamma$ and are broken
  once: ``starting at an end of $\gamma$'' is the upper degeneracy, and the
  break is the lower one,
\item
  trajectories that start at an end of $\gamma$ (i.e. at $\star$, seen
  either as $\gamma(0)$ or $\gamma(1)$) and have zero length: ``starting
  at an end of $\gamma$'' is the upper degeneracy, and ``having zero
  length'' is the lower one,
\end{itemize}

At each such configuration, one can resolve either the upper or the lower
degeneracy, preserving the other one. For convenience, we call the
resolution of an upper or lower degeneracy an upper or lower ``gluing'',
even if the degeneracy is not actually a trajectory break.

Such an operation selects a $1$-dimensional arc in the boundary of
$\cM(\gamma,M)$, whose other end is again a point in $B$. In other words,
the boundary components of $\cM(\gamma,M)$ that contain corners form a
graph, whose vertices are the elements of $B$ and do all have a valency of
$2$, and whose edges correspond to upper or lower gluings.

In particular, this graph is a collection of loops, and each of them
defines (up to orientation) a cyclic sequence of elements of $B$ related
by alternating upper and lower gluings.

We call such a sequence the \emph{crocodile walk} along the corresponding
boundary component of $\partial\cM(\gamma,M)$. Arcs that correspond to
upper gluings are called upper steps, and arcs associated to lower
gluings are called lower steps.

\medskip

We are interested in the boundary component of $\cM(\gamma,M)$ that
contains the zero length trajectory $\{\star\}$, seen as $\gamma(0)$.

The ``upper gluing'' in this case, consists in relaxing the condition of
starting at an end of $\gamma$ while keeping that of having zero length:
it is nothing but the path $\gamma$ itself, seen as a one parameter of
zero length trajectories, and its other end is the zero length trajectory
at $\star$, now seen as $\gamma(1)$.

In the other direction, the walk associated to this boundary component
starts with a lower gluing, and gives an alternating sequence of upper
and lower gluings
\begin{equation}
  \label{eq:crocowalkdetail}
  \begin{tikzcd}[start anchor=east,end anchor=west]
    \llap{$\{\star\}=\,$}
    u_{1}\ar[r,bend right=20,"\ell_{1}",pos=.7]&
    u_{2}\ar[r,bend left =20,"h_{1}"]&
    u_{3}\ar[r,bend right=20,"\ell_{2}"]&
    u_{4}\ar[r,phantom,"\dots"]&
    u_{2N-1}\ar[r,bend right =20,"\ell_{N}",pos=.8]&
    u_{2N}\rlap{$\,=\{\star\}$}
  \end{tikzcd}
\end{equation}
until the zero length trajectory at $\gamma(1)$ is reached. Here, each
$u_{i}$ is a configuration in $B$, and each edge $\ell_{i}$ is a lower
gluing.%

In particular, $l_{i}$ is a family of trajectories of the form
$\lambda_{i}\cdot\sigma_{i}$, where $\lambda_{i}$ is a fixed trajectory
that is either a trajectory from $\gamma$ to an index $1$ critical point
$y_{i}$ or the zero length trajectory at $\star$, and $\sigma_{i}$ is a
Morse step through $y_{i}$ or $\star$. Observe moreover that the Morse
steps in the sequence $\sigma_{1}\dots\sigma_{N}$ are consecutive.

\begin{definition}\label{def:croco walk}
With the above notations, define the Morse loop associated to $\gamma$ by
the crocodile walk process as
\begin{equation}\label{eq:pushdown}
  \Theta(\gamma)=\sigma_{1}\dots\sigma_{N}.
\end{equation}
\end{definition}

\begin{figure}
  \centering
  \begin{tikzpicture}[xscale=.5,yscale=1,out=rand*20, in=180+rand*20, relative]

    \draw(0,3)--++(10,0);
%    \node[left ] at ( 0,3) {$\wstar$};
%    \node[right] at (10,3) {$\wstar$};
    \pgfmathsetseed{10}%

    \firststep{0}
    \upperstarstepright{0}
    \lowerstep{1}
    \upperstep{2}
    \lowerstep{3}
    \upperstarstepleft{4}
    \lowerstarstep{5}
    \upperstarstepright{6}
    \lowerstep{7}
    \upperstarstepleft{8}
    \laststep{9}

    \path( 0,3) node[above]{$\gamma(0)$};%
    \path(10,3) node[above]{$\gamma(1)$};
  \end{tikzpicture}
  \caption{Crocodile walk associated to a loop $\gamma$.}
  \label{fig:crocowalk}
\end{figure}
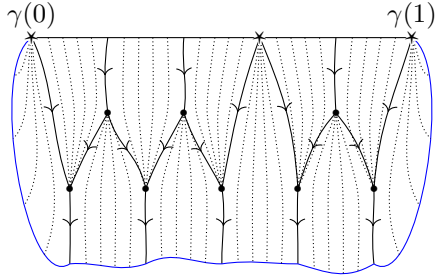

Theorem \ref{thm:surjectivity} is a direct consequence of the following proposition:
\begin{proposition}\label{prop:croco homotopy}
  With the above notations, $\gamma$ and $\ev(\Theta(\gamma))$ are
  homotopic as topological based loops.
\end{proposition}

\begin{proof}
  We can use the value of $-f$ as a parametrization of the trajectories
  in $\cM(y,M)$. Moreover, for compactness reason, the lengths of the
  corresponding intervals are uniformly bounded. This means that we can
  find a uniform reparametrization $\theta$ of all the trajectories by
  $[0,1]$. As a consequence, we obtain an evaluation map of the form
  $$
  \begin{tikzcd}[row sep =1ex]
    \cM(y,M)\times[0,1]\arrow[r]& M\\
    (u,t)\arrow[r,mapsto]&\pi(u(\theta_{u}(t)))  
  \end{tikzcd}.
  $$
  We pick such choices for all the moduli spaces involved in the
  crocodile walk. Since a step $\ell_{i}$ or $h_{i}$ is (the product of a
  constant trajectory with) a component of a moduli space of the form
  $\cM(y,M)$ or $\cM(\star,M)$  that is homeomorphic to a segment, it can be
  parametrized by $[0,1]$, and the above evaluation induces a collection
  of maps
  $$
  \begin{tikzcd}[row sep =1ex]
    \Delta_{i}=[0,1]^{2}\arrow[r,"\ev_{\ell_{i}}"]& M
  \end{tikzcd},\quad
  \begin{tikzcd}[row sep =1ex]
    \Delta'_{i}=[0,1]^{2}\arrow[r,"\ev_{h_{i}}"]& M
  \end{tikzcd}.
  $$
  Finally,
  $$
  {\ev_{\ell_{i}}}_{|_{\{1\}\times[0,1]}}={\ev_{h_{i}}}_{|_{\{0\}\times[0,1]}}
  \text{ and }
  {\ev_{h_{i}}}_{|_{\{1\}\times[0,1]}}={\ev_{\ell_{i+1}}}_{|_{\{0\}\times[0,1]}}
  $$
  so that these squares patch side to side to form a wider square, whose
  evaluation to $M$ is
  \begin{itemize}
  \item
    constant equal to $\star$ on the left and right sides,
  \item
    the evaluation $\ev(\Theta(\gamma))$ on the lower side,
  \item
    factors through $\gamma$ on the upper side, by a map that sends $0$
    to $0$ and $1$ to $1$. 
  \end{itemize}
  The upper side hence covers the loop $\gamma$ with degree $1$, and this
  square defines an homotopy from $\gamma$ to
  $\ev(\Theta(\gamma))$.
\end{proof}

\subsection{Upward crocodile walk.}\label{sec:upcroco}

Finally, in all the above construction, up and down, start and end, can
be swapped, and we can do the same work in $\cM(M,\gamma)$ instead of
$\cM(\gamma,M)$. The boundary of this space can be walked through in in
the very same way as before, and we speak of the upward crocodile walk. 

Equivalently, we could replace $f$ by $-f$ (which also replaces the
quadratic form at infinity by its opposite, and hence swaps $\R^{N_{+}}$
and $\R^{N_{-}}$) in the above construction, but still remember the
direction of the trajectories.

Notice that in this version, the role previously played by index
$0$ and $1$ points is now played by index $n$ and $n-1$ points.

\bigskip

Consider now a generic loop in $M$, and both downward and upward
crocodile walks on this loop. Each of them is a one parameter family of
trajectories in $\cM(\gamma,M)$ and $\cM(M,\gamma)$ that split as a
sequence of upper and lower steps. Recall that as a fiber product the
spaces $\cM(\gamma,M)$ and $\cM(M,\gamma)$ are endowed with projection to
$[0,1]$, which is the source of $\gamma$. Observe moreover that during a
lower (resp. upper) step of the downward (resp. upward) crocodile walk,
this projection is constant. Call the corresponding point on $\gamma$ the
attachment point of the corresponding step.

Then, for a generic choice of Morse data and loop $\gamma$, the attaching
points of lower and upper steps are all different, except for the steps
through the base point, where it is always $\star$ by definition.

This special behaviour makes the steps through the base point a particular
case, that may require special attention in all the subsequent
definitions. To avoid this and keep things uniform, we prefer to regard
this as a lack of genericity, due to the fact we use the same base point
for upward and downward crocodile walks, and solve this we fix another
base point $\ostar$, that is be a small perturbation of $\star$, to be
used as the base point for upward crocodile walks.

Notice that transversality with $\cM(M,M)\xto{\ev_{+}}M$ and
$\cM(M,M)\xto{\ev_{-}}M$ for the base point is an open condition in our context,
we can choose $\ostar$ in such a way that there is a path $\delta$ from
$\star$ to $\ostar$ such that all its points remain transverse to both
$\ev_{-}$ and $\ev_{+}$ on $\cM(M,M)$.

\begin{definition}\label{def:costep}
A Morse co-step is a component with boundary of a moduli space $\cM(b,M)$
for some $b\in\Crit_{n-1}(f)$, or of $\cM(\ostar,M)$.  
\end{definition}

The upward crocodile walk then turns a (generic) loop $\gamma$ based at
$\ostar$ into a sequence
\begin{equation}
  \label{eq:upwardcrocodile}
  \Xi(\gamma)=\tau_{1}\cdots \tau_{k}  
\end{equation}
of consecutive Morse co-steps, from $\ostar$ to $\ostar$, such that its
evaluation in $M$ is homotopic to $\gamma$.

\section{Relations}

\subsection{Relation patch}

%Recall the projection
%$$
%\bM(a,M)\xto{\pi_{1}}\cM(a,M)
%$$
%obtained by forgetting the part of the trajectory that is below the bounce.
%\textcolor{red}{Pourquoi est-ce qu'on rappelle ce point ici ?}

Consider a sequence $\sigma=(\sigma_{1}\dots\sigma_{k})$ of consecutive
Morse steps. To fix notations, consider that $\sigma_{i}$ starts from
$(\beta_{i},\alpha_{i})\in\cM(y_{i},x_{i})\times\cM(x_{i},M)$ and ends at
$(\beta'_{i},\alpha'_{i})\in\cM(y_{i},x'_{i})\times\cM(x'_{i},M)$. 
Because the steps are consecutive, we have, for $i\in\{1,\dots,k-1\}$:
$$
x'_{i}=x_{i+1} \quad\text{ and }\quad \alpha'_{i}=\alpha_{i+1}.
$$
Notice that we can have $y_{i}=\star$, in which case $\beta_{i}$ (resp.
$\beta'_{i}$) is the $0$-length trajectory at $\star$ or a trajectory
from $\star$ to $x_{i}$ (resp. $x'_{i}$).
\begin{definition}
  Consider an index $n$ critical point $a\in\Crit_{n}(f)$. A path
  $[0,1]\xto{\lambda}\partial\bM(a,M)$ is said to extend a sequence of
  consecutive steps $\sigma=(\sigma_{1}\dots\sigma_{k})$ if there are
  times $0=t_{0}<\dots<t_{2k+1}=1$ such that
  \begin{itemize}
  \item
    on $[t_{2i},t_{2i+1}]$, $\lambda(t)$ is a path of trajectories that
    are all broken at $x_{i+1}$ and end with $\alpha_{i+1}$, 
  \item
    on $[t_{2i-1},t_{2i}]$, $\lambda(t)$ is a path of the form
    $\{u\}\times\sigma_{i}$, where $u$ is a fixed trajectory in
    $\bM(a,y_{i})$.
 \end{itemize}
\end{definition}

\begin{figure}
  \centering
  \begin{tikzpicture}[%
    relative=true,out=rand*10,in=180+rand*10,%
    decoration={amplitude=25pt,segment length=40pt},%
    croco/start=midup,%
    croco/end=midup,%
    croco/save nodes=true,%
    croco/save node prefix=AA,%
    croco/proba=.75,%
    ]
    \pgfmathsetseed{3}
    
    \path (0,4)coordinate(a)node{\tiny$\bullet$}node[above]{$a$};
    \draw[croco]
    (-3,0)--(3,0);
    \foreach \i [evaluate=\i as \t using \i/\thecroconodenumber]
       in {0,...,\thecroconodenumber}{
      \draw[trajectory](a)to[relative=false, out=-140+\t*100, in=90](AA\i);
    }
    \path($(AA1)+(0,-1.6)$)node{$\sigma_{1}$};
    \path($(AA2)+(0,-1.6)$)node{$\sigma_{2}$};
    \path($(AA3)+(0,-1.6)$)node{$\sigma_{3}$};
    \path($(AA4)+(0,-1.6)$)node{$\dots$};
    \path($(AA5)+(0,-1.6)$)node{$\sigma_{k}$};
    \end{tikzpicture}  
  \caption{A path in $\bM(a,M)$, based on a sequence of Morse steps.}
\end{figure}
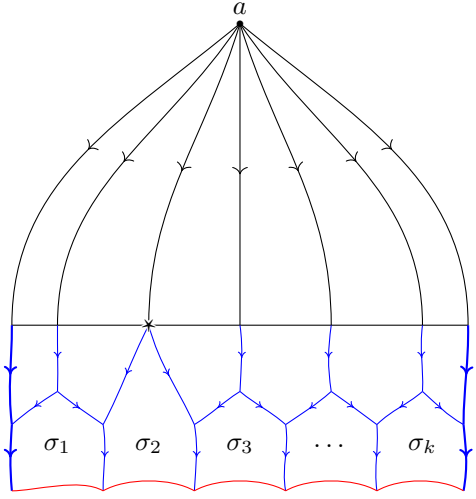

\begin{remark}
  This notion, of a path of trajectories ``extending'' a sequence of
  Morse steps, can obviously be defined in other moduli spaces. For
  instance, the construction of the crocodile walk along a loop $\gamma$
  in $M$ can be seen as a path in $\cM(\gamma,M)$ that extends the
  corresponding Morse loop.
\end{remark}

When a path $\lambda$ extends a sequence of consecutive Morse steps, we
will refer to the sequence of Morse steps as ``the lower part'' of
$\lambda$.

%%For a generic choice of $\lambda$, the
%%trajectories that go through an index $1$ critical point or $\star$ after
%%the bounce are isolated, and in particular, this does not happen at the
%%ends of $\gamma$. However, when playing with upward and downward
%%crocodile walks, one has to face situations where $\lambda$ starts and
%%ends with trajectories that go through the base point. As already
%%mentioned in section \ref{sec:upcroco}, this leads to an annoying lack of
%%genericity, and to avoid it all at once, we choose to perturb the base
%%point $\star$ into another generic point $\ostar$, that will be used for
%%upward crocodile walks. 

\begin{definition}\label{def:patch}
  A relation patch is a couple $(a,(\lambda_{0},\dots,\lambda_{N}))$
  where $a\in\Crit_{n}(f)$ is an index $n$ critical point, and
  $(\lambda_{0},\dots,\lambda_{N})$ is a cyclic sequence of paths in
  $\bM(a,M)$ such that
  \begin{itemize}
  \item
    each $\lambda_{i}$ is an extension of a sequence of consecutive Morse
    steps,
  \item
    except when $N=0$ and $\lambda_{0}$ is a loop, for each $i$, there is
    a trajectory $u_{i}$ from $a$ to some $n-1$ critical point $b_{i}$ or
    $\ostar$, such that the trajectories at the end of $\lambda_{i}$ and
    at the start of $\lambda_{i+1}$ both start with $u_{i}$, i.e. take
    the form
    \begin{align*}
      \lambda_{i}(1)&=u_{i}\cdot v\cdot\alpha&
      \lambda_{i+1}(0)&=u_{i}\cdot v'\cdot\alpha'
    \end{align*}
    where $\alpha$ (resp. $\alpha'$) is the end (resp. start) of the
    last (resp. first) step $\lambda_{i}$ (resp. $\lambda_{i+1}$) is an
    extension of (see figure \ref{fig:patch}).
  \end{itemize}
  Moreover we call the set $((v,\alpha),(v',\alpha'))$ formed by the
  parts of the end of $\lambda_{i}$ and the start of $\lambda_{i+1}$ that
  are below their shared trajectory $u_{i}$ an \emph{open edge} of the
  patch $(a,(\lambda_{0},\dots,\lambda_{N}))$.
\end{definition}

%\begin{remark}
%  In the particular case where $\lambda=(\lambda_{0})$ is a single path
%  and this path is in fact a loop, then no compatibility condition is
%  required at the end or beginning of $\lambda$, and the patch has no
%  open edge.   
%\end{remark}

\begin{remark}
  In practice, we will only need to work with patches having 0,1, or 2 open
  edges.
\end{remark}

\begin{figure}
  \centering
  \begin{tikzpicture}[
    scale=.6,
    relative=true,out=rand*10,in=180+rand*10,
    decoration={amplitude=25pt,segment length=20pt},
    croco/start=midup,
    croco/end=midup,
    croco/save nodes=true,
    croco/save node prefix=AA,
    croco/proba=.75,
    ]
    
    \pgfmathsetseed{7}%
    \begin{scope}[shift={(-5.5,0)}]      
      \begin{pgfonlayer}{foreground}
        \path(0,0)coordinate(a)node{\tiny$\bullet$}node[above]{\tikzoutline{$a$}};
      \end{pgfonlayer}
      \path(  0:2)coordinate(b0)node{\tiny$\bullet$}node[right]{$b_{0}$};
      \path(180:2)coordinate(b1)node{\tiny$\bullet$}node[left]{$b_{1}$};
      \path(b0)+( -60:2)coordinate(c0);%node{\tiny$\bullet$}node[right]{$c_{0}$};
      \path(b0)+(  60:2)coordinate(c1);%node{\tiny$\bullet$}node[right]{$c_{1}$};
      \path(b1)+( 120:2)coordinate(c2);%node{\tiny$\bullet$}node[left]{$c_{2}$};
      \path(b1)+(-120:2)coordinate(c3);%node{\tiny$\bullet$}node[left]{$c_{3}$};
      
      \draw[trajectory] (a)to(b0);
      \draw[trajectory] (a)to(b1);
      \draw[trajectory](b0)to(c0);
      \draw[trajectory](b0)to(c1);
      \draw[trajectory](b1)to(c2);
      \draw[trajectory](b1)to(c3);

      \path[croco](c3)to[out=-20,in=-160](c0);
      \path[remember points,
        remember points/count=(\thecroconodenumber+1),
        remember points/prefix=BB](b1)to[out=-60,in=-120](b0);
      \addtocounter{croconodenumber}{-1}
      \foreach \i in{1,...,\thecroconodenumber}{
        \draw[thin, densely dotted](a) ..controls (BB\i)..(AA\i);
      }

      \path[croco](c1)to[out=-20,in=-160](c2);
      \path[remember points,
        remember points/count=(\thecroconodenumber+1),
        remember points/prefix=BB](b0)to[out=-60,in=-120](b1);
      \addtocounter{croconodenumber}{-1}
      \foreach \i in{1,...,\thecroconodenumber}{
        \draw[thin, densely dotted](a) ..controls (BB\i)..(AA\i);
      }

      \node[below] at ($(a)!.5!(b0)$){$u_{0}$};
      
      \path(b0)--++(-60:2)%node[above right]{$v'_{0}$}
      --++(-75:2.2)coordinate(c)--++(-70:1.5)coordinate(d);
      \draw[decoration={brace,raise=12pt,amplitude=3pt},decorate]
      ($(b0)!.1!(c)$)--(c)node[pos=.5,above right=0pt and 13pt]{$v$};
      \draw[decoration={brace,raise=12pt,amplitude=3pt},decorate]
        (c)--(d)node[pos=.5,above right=0pt and 13pt]{$\alpha$};
        
      \path(b0)--++(60:2)--++(75:2.2)coordinate(c)--++(70:1.5)coordinate(d);
      \draw[decoration={brace,raise=12pt,amplitude=3pt,mirror},decorate]
        ($(b0)!.1!(c)$)--(c)node[pos=.5,below right=0pt and 13pt]{$v'$};
      \draw[decoration={brace,raise=12pt,amplitude=3pt,mirror},decorate]
        (c)--(d)node[pos=.5,below right=0pt and 13pt]{$\alpha'$};

        \draw[->]($(b1)+(-110:6)$)to[out=-20,in=180+20]
        node[pos=.5,below]{$\lambda_{0}$}
        ($(b0)+(-70:6)$);

        \draw[<-]($(b1)+( 110:6)$)to[out=20,in=180-20]
        node[pos=.5,above]{$\lambda_{1}$}
        ($(b0)+( 70:6)$);
    \end{scope}

    \pgfmathsetseed{35}%
    \begin{scope}[shift={(5.5,0)}]
      \def\tmpangle{20}
      \path(0,0)coordinate(a)node{\tiny$\bullet$}node[above]{$a$};
      \path(  0:2)coordinate(b0)node{\tiny$\bullet$}node[right]{$b_{0}$};
      \path(b0)+( -60:2)coordinate(c0);
      \path(b0)+(  60:2)coordinate(c1);
%      \path( 180-\tmpangle:3)coordinate(c2)node{$\wstar$};
%      \path(-180+\tmpangle:3)coordinate(c3)node{$\wstar$};
      \path( 180-\tmpangle:3)coordinate(c2);
      \path(-180+\tmpangle:3)coordinate(c3);

      \coordinate(b1)at($(a)!.5!(c2)$);
      \coordinate(bb1)at($(a)!.5!(c3)$);
      \draw[trajectory] (a)to(b0);
%      \draw[trajectory] (a)to(b1);
      \draw[trajectory](b0)to(c0);
      \draw[trajectory](b0)to(c1);
      
      \draw[trajectory,out=5,in=180-5] (a)to(c2);
      \ifnum\tmpangle>0{
        \draw[trajectory,out=5,in=180-5] (a)to(c3);
        \foreach \t[parse=true] in {1,2,...,(floor(\tmpangle/2))}{
          \draw[very thin,out=5,in=180-5] (a)to($(c2)!\t/(floor(\tmpangle/2)) !(c3)$);
        }
        \begin{pgfonlayer}{foreground}
          \draw[very thick](c2)--(c3);
          \node at (c2){\small$\wostar$};
          \node at (c3){\small$\wostar$}; 
        \end{pgfonlayer}
      }\fi
    
%      \path[croco,croco/start=star](c3)to[out=-20,in=-160](c0);
      \path[croco,croco/start=midup](c3)to[out=-20,in=-160](c0);
      \path[remember points,
        remember points/count=(\thecroconodenumber+1),
        remember points/prefix=BB](bb1)to[out=-60,in=-120](b0);
      \foreach \i [parse=true] in{1,...,\thecroconodenumber-1}{
        \draw[thin, densely dotted](a) ..controls (BB\i)..(AA\i);
      }
        
%      \path[croco,croco/end=star](c1)to[out=-20,in=-160](c2);
      \path[croco,croco/end=midup](c1)to[out=-20,in=-160](c2);
      \path[remember points,
        remember points/count=(\thecroconodenumber+1),
        remember points/prefix=BB](b0)to[out=-60,in=-120](b1);
      \foreach \i [parse=true] in{1,...,\thecroconodenumber-1}{
        \draw[thin, densely dotted](a) ..controls (BB\i)..(AA\i);
      }

      \draw[->]($(b1)+(-115:6)$)to[out=-20,in=180+20]
      node[pos=.5,below]{$\lambda_{0}$}
      ($(b0)+(-75:6)$);
      
      \draw[<-]($(b1)+( 127:5.3)$)to[out=20,in=180-20]
      node[pos=.5,above]{$\lambda_{1}$}
      ($(b0)+( 75:6)$);

    \end{scope}
  \end{tikzpicture}
  \caption{Two relation patches, with 2 ``open edges''
    (the non red parts of the boundary).
    For convenience, the single trajectory from $a$ to $\ostar$ is
    repeated on the drawing to form an edge.}
  \label{fig:patch}

\end{figure}
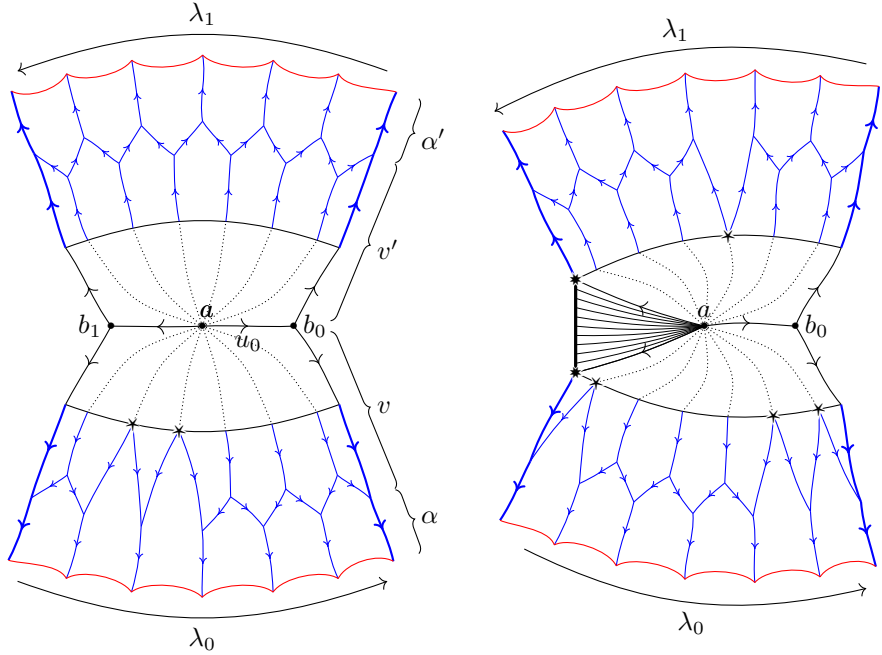

\begin{proposition}\label{prop:PatchDisc}
  Given a patch $(a,(\lambda_{0},\dots,\lambda_{N}))$ as in definition
  \ref{def:patch}, there is a continuous $2$-disc
  $\Delta\xrightarrow{\ev} M$ whose boundary splits as a cyclic sequence
  of intervals, where $\ev$ is alternately
  \begin{itemize}
  \item
    the evaluation in $M$ of the sequence of Morse steps forming the
    bottom part of $\lambda_{i}$,
  \item
    the projection in $M$  of the open edge between $\lambda_{i}$ and
    $\lambda_{i+1}$.
  \end{itemize}
\end{proposition}

\begin{proof}
  All the trajectories in $\bM(a,M)$ are parametrized by compact
  intervals, whose length is continuous, and bounded away from $0$. In
  particular, they can be uniformly reparametrized, to obtain an
  evaluation map
  $$
  \bM(a,M)\times[0,1]\to M.
  $$

  Given a patch as in definition \ref{def:patch}, each path $\lambda_{i}$
  induces an evaluation from a square $[0,1]\times[0,1]$ to $M$.  Notice
  that all these trajectories start at $a$, so that this evaluation
  factors through the quotient $[0,1]\times[0,1]/[0,1]\times\{0\}$.

  Moreover, we can fix parametrization of the trajectories $(u_{i})$
  that are shared by consecutive paths and manage to have them appear on
  the first halves $\{1\}\times[0,\tfrac{1}{2}]$ and
  $\{0\}\times[0,\tfrac{1}{2}]$ of the relevant sides. Gluing the
  rectangles along these half sides, and collapsing $[0,1]\times\{0\}$ to
  a point, we obtain a disc $\Delta_{a,\lambda}$ associated to the
  initial patch, endowed with an evaluation map to $M$ satisfying the
  desired properties.
\end{proof}

\begin{definition}
  Consider two relation patches $(a_{0},\lambda)$ and $(a_{1},\mu)$, and
  choose an open edge $((v_{0},\alpha_{0}),(v'_{0},\alpha'_{0}))$ and
  $((v_{1},\alpha_{1}),(v_{1}',\alpha'_{1}))$ on each of them. The two
  patches are said to match along these open edges if
  \begin{equation}
    \label{eq:patchmatch}
    (v_{0},\alpha_{0})=(v'_{1},\alpha'_{1})
    \quad\text{ and }\quad
    (v'_{0},\alpha'_{0})=(v_{1},\alpha_{1}).
  \end{equation}
  (see figure \ref{fig:patchmatch}).
\end{definition}

\begin{remark}\label{rem:ConcatAtMatchingEdges}
  Let $\lambda_{i}$ be the path at the end of which
  $(v'_{0},\alpha'_{0})$ appears in the first patch, and
  $\sigma=(\sigma_{i,1},\dots,\sigma_{i,k_{i}})$ be the sequence of Morse
  steps $\lambda_{i}$ is based on. Similarly, let $\mu_{j}$ be the path at the start of which
  $(v_{1},\alpha_{1})$ appears in the second patch, and
  $\tau=(\tau_{j,1},\dots,\tau_{j,l_{j}})$ be the sequence of Morse steps
  $\mu_{j}$ is based on.

  Then it is a consequence of the definition that $\sigma_{i,k_{i}}$ and
  $\tau_{j,1}$ are consecutive steps, so that the concatenated sequence
  $$
  \sigma_{i,1}\cdots\sigma_{i,k_{i}}\cdot \tau_{j,1}\cdots\tau_{j,l_{j}}
  $$
  is still a sequence of consecutive Morse steps. 
\end{remark}

\begin{definition}\label{def:tree of matching patches}
  A tree of matching patches is a finite planar tree $T$, with one
  relation patch $(a,\lambda)$ on each vertex, such that to each edge is
  associated to a pair of matching open edges, such that
  \begin{enumerate}
  \item
    all the open edges match by pairs,
  \item
    at each vertex, the planar structure of the tree and of the patch
    agree.
  \end{enumerate}
\end{definition}

\begin{remark}
  In fact, we are only interested in the case where the tree $T$ is a
  line. In this case, we may speak of a line of matching patches. 
\end{remark}

\begin{remark}
  As already observed in remark \ref{rem:ConcatAtMatchingEdges}, matching
  open edges allow to concatenate the corresponding sequences of Morse
  steps. In a tree of matching patches, since all the open edges match by
  pairs, the concatenation of all these sequences define a Morse (free)
  loop. Moreover, gluing the domains $\Delta_{a,\lambda}$ associated to
  each patch (see \ref{prop:PatchDisc}) along matching edges, we obtain a
  domain homeomorphic to a planar tubular neighbourhood of the initial
  tree, which is a disc. Moreover, it is endowed  with an evaluation map
  to $M$, whose restriction to the boundary is the evaluation of the
  Morse loop.
\end{remark}

\begin{proposition}\label{prop:evaltree}
  Given a tree of matching relation patches, the concatenation of all the
  the bottom parts of the patches according to the matching open edges
  form a Morse free loop, whose evaluation in $M$ is trivial.

  This loop will still be called the basis, or the bottom part, of the tree
  of patches.
\end{proposition}

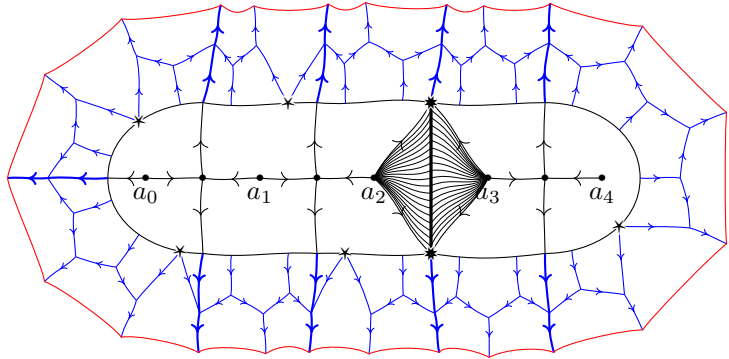
\begin{figure}
  \centering
  \begin{tikzpicture}[
    relative,out=rand*10,in=180+rand*10,
    decoration={amplitude=15pt,segment length=25pt}]
    
    \pgfmathsetseed{5}
    \foreach \i in {0,...,4}{
      \path (1.5*\i,0) coordinate (a\i) node{\tiny$\bullet$}node[below]{$a_{\i}$};
    }
    \foreach \i[evaluate=\i as \ii using \i+1] in {0,...,3}{
      \coordinate (b\i) at ($(a\i)!.5!(a\ii)$);
      }
    \coordinate(b-1)at(-.5,0);
    \coordinate(b4) at ($(a4)+(.5,0)$);
    \foreach \i in {0,1,3}{
      \path (b\i) node{\tiny$\bullet$};%node[above left]{$b_{\i}$};
      }
    \foreach \i[evaluate= \i as \ii using \i+1] in {0,...,3}{
      \coordinate(c\i0) at ($(b\i)+(0,-1)$);
      \coordinate(c\i1) at ($(b\i)+(0, 1)$);
    }

    \foreach \i[evaluate= \i as \ii using \i+1] in {0,1,3}{
      \draw[trajectory](a\ii)to(b\i);
      \draw[trajectory](a\i)to(b\i);
      \draw[trajectory](b\i)to(c\i0);
      \draw[trajectory](b\i)to(c\i1);
    }
    \draw[trajectory](a2)to[out=10,in=180+10](c20);
    \draw[trajectory](a2)to[out=10,in=180+10](c21);
    \draw[trajectory](a3)to[out=10,in=180+10](c20);
    \draw[trajectory](a3)to[out=10,in=180+10](c21);

    \begin{pgfonlayer}{foreground}
      \draw[very thick](c20)--(c21);
      \node at (c20) {$\wostar$};
      \node at (c21) {$\wostar$};
    \end{pgfonlayer}

    \foreach \t in {.05,.1,...,.96}{
      \draw[ultra thin](a2)to[out=15,in=180+15]($(c20)!\t!(c21)$);
      \draw[ultra thin](a3)to[out=15,in=180+15]($(c20)!\t!(c21)$);
    }
    
    \draw[trajectory](a0)to(b-1);

    \begin{scope}[croco/start=continue,croco/end=midup]
%      \draw[croco,croco/start=star,
      \draw[croco,croco/start=midup,
      relative=false, out=-90,in=180,
      ](b-1)to(c00);
      \draw[croco](c00)to(c10);
%      \draw[croco,croco/end=star](c10)to(c20);
      \draw[croco,croco/end=midup](c10)to(c20);
      \draw[croco](c20)to(c30);
      \draw[croco,relative=false](c30)to[in=-90, out=0] (b4) to[out=90, in=0](c31);
%      \draw[croco,croco/end=star](c31)to(c21);
      \draw[croco,croco/end=midup](c31)to(c21);
      \draw[croco](c21)to(c11);
      \draw[croco](c11)to(c01);
      \draw[croco,croco/end=close,
      relative=false, in=90, out=180](c01)to(b-1);
    \end{scope}

  \end{tikzpicture}
  \caption{A line of matching patches.}
  \label{fig:patchmatch}
\end{figure}

\subsection{Generators for the relations}
We abuse notation and for $\gamma\in\MSLoops(\Xi,\star)$, we write $\Theta(\gamma)$
for $\Theta(\ev(\gamma))$ (see definition \ref{def:croco walk}).

We define the following relations on $\MSLoops(\Xi,\star)$:
\begin{enumerate}
\item
  $\gamma \sim \Theta(\gamma)$ for $\gamma \in \MSLoops(\Xi,\star)$, (see definition
  \ref{def:croco walk}),
%\item
%  $\Theta(\gamma_{1}\cdot\gamma_{2}) \sim \Theta(\gamma_{1}) \cdot
%  \Theta(\gamma_{2})$, for $\gamma_{1}, \gamma_{2} \in
%  \MSLoops(\Xi,\star)$
\item
  $w\gamma w^{-1}\sim1$ where $\gamma$ is the bottom part of a tree of
  matching patches, and $w$ is a Morse path from $\star$ to the start of
  $\gamma$,
\end{enumerate}
and let $\cR=\cR(\Xi,\star)$ denote the normal subgroup generated by these relations.

\begin{remark}  
  Notice that in the (not stable) Morse setting, relations of type
  $(1)$ are trivial, while relations of type $(2)$ are
  generalizations to arbitrary index of the usual relations associated to
  index $2$ critical points. 

  Relation of type $(1)$ naturally appear when considering functoriality
  properties of the construction. Even in the (not stable) Morse
  setting, the natural way to define a morphism between the groups
  associated to two different Morse functions, is to flow down the loops
  associated to the first function along the gradient of the second. From
  our point of view, this means applying the crocodile walk associated to
  the second function to the loops associated to the first function.

  In the particular case where the two functions are the same (or small
  perturbations of the same function) this operation should lead to the
  identity. In the (not stable) Morse setting, this operation happens
  to be the identity at the loop level, i.e. before taking any relation,
  into account. In general, $\gamma$ and $\Theta(\gamma)$ are different,
  and this identification is required.
\end{remark}

\begin{proposition}
  Relations do all evaluate to the trivial class in $\pi_{1}(M,\star)$.  
\end{proposition}
\begin{proof}
As already observed, $\gamma\in\MSLoops(\Xi,\star)$ and $\Theta(\gamma)$
have homotopic evaluations in $M$,
%As a consequence, this also holds for
%$\Theta(\gamma_{1}\cdot\gamma_{2})$ and
%$\Theta(\gamma_{1})\cdot\Theta(\gamma_{2})$.
and a direct consequence of
proposition \ref{prop:evaltree} is that an evaluation of a relation
associated to a tree of matching patches is trivial. Finally, since the
evaluation is trivial on all the generators, it is also trivial on the
generated normal subgroup $\cR$. 
\end{proof}

As a consequence, the evaluation map to $\pi_{1}(M,\star)$ is well
defined on the quotient $\MSLoops(\Xi,\star)/\cR$.
\begin{theorem}\label{thm:pi1}
  The evaluation defines a group isomorphism:
  $$
  \begin{tikzcd}
    \MSLoops(\Xi,\star)/\cR
    \arrow[r,"\ev","\sim"']&\pi_{1}(M,\star)
  \end{tikzcd}.
  $$
\end{theorem}

Surjectivity is a direct consequence of proposition \ref{prop:croco homotopy}. The end of this section is dedicated to the proof of
injectivity.

\subsection{Preparation lemmas}

The crocodile walk process turns a topological loop into a Morse loop by
means of a one parameter family of trajectories that start on the
topological loop, and end in $M$. However, the starting point of these
trajectories does not need to run injectively from the start to the end
of the topological loop: in general, it starts at the start, and ends at
the end, but may go back and forth along the loop in between.

The construction to come to prove theorem \ref{thm:pi1} will be playing with downward
and upward crocodile walks. In this section, we give technical lemmas
that allow the required reparametrization to ``synchronize'' them.

\begin{lemma}\label{lem:fiberproductofintervals}
  Let $\alpha:I_{\alpha}=[0,1]\to[0,1]$ and
  $\beta:I_{\beta}=[0,1]\to[0,1]$ be continuous and piecewise smooth maps,
  with no common singular value, sending $0$ to $0$ and $1$ to $1$.

  Then there are continuous maps $\phi_{\alpha}$ and $\phi_{\beta}$,
  sending $0$ to $0$ and $1$ to $1$, making the following diagram
  commutative:
  \begin{equation}
    \label{eq:fiber product of intervals}
    \begin{tikzcd}
      {[0,1]}
      \arrow[dr]
      \arrow[drr,bend left=10,"\phi_{\beta}"]
      \arrow[rdd,bend right=10,"\phi_{\alpha}"']
      \\
      &
      I_{\alpha}\fibertimes{\alpha}{\beta} I_{\beta}
      \arrow[r,"\pi_{2}"]\arrow[d,"\pi_{1}"]&
      I_{\beta}\arrow[d,"\beta"]
      \\
      &I_{\alpha}\arrow[r,"\alpha"]&{[0,1]}
    \end{tikzcd}.
  \end{equation}
\end{lemma}

\begin{remark}
  Continuous piecewise smooth and smooth make no difference here as an
  injective reparametrization allows to turn the former into the latter.
  In practice however, these maps will come from evaluation of crocodile walks
  on the initial loop: they are typically continuous and piecewise
  smooth (they are constant on every other interval of some subdivision).
\end{remark}

\begin{proof}
  The fiber product $I_{\alpha}\fibertimes{\alpha}{\beta}I_{\beta}$ is a
  piecewise smooth compact $1$ dimensional manifold, with boundary living
  above $0$ and $1$ (for the projection
  $\alpha\circ\pi_{1}=\beta\circ\pi_{2}$) (see figure
  \ref{fig:pullbacklemma}). If $0$ is a regular value of $\beta$ for
  instance, then the pre-image in
  $I_{\alpha}\fibertimes{\alpha}{\beta}I_{\beta}$ of a small
  neighbourhood $[0,\epsilon)\subset[0,1]$ of $0$ by $\beta\circ\pi_{2}$
  is homeomorphic to the pre-image in $I_{\alpha}$ of some $[0,\epsilon')$
  by $\alpha$, and has therefore exactly one component that has boundary
  above $0$.
%
%  (\textcolor{red}{or: the intersection $I_{\alpha}
%    \fibertimes{\alpha}{\beta} I_{\beta}\cap[0,1]\times [0,\epsilon)$ is
%    homeomorphic to $\Gamma_\alpha\cap[0,1]\times [0,\epsilon')$, and
%    since this latter has only $(0,0)$ as a boundary point we deduce that
%    $I_{\alpha}\fibertimes{\alpha}{\beta}I_{\beta}$ has exactly one
%    component etc.}).
%

  The same holds above $1$, and by compactness we conclude that exactly one
  component $J$ of $I_{\alpha}\fibertimes{\alpha}{\beta}I_{\beta}$ has
  boundary, and that this boundary has one point above $0$ and the other
  one above $1$. Choosing a parametrization $[0,1]\xto{\theta}J\subset
  I_{\alpha}\fibertimes{\alpha}{\beta}I_{\beta}$ of this component, we
  obtain maps $\phi_{\beta}=\pi_{1}\circ\theta$ to $I_{\alpha}$ and
  $\phi_{\alpha}=\pi_{2}\circ\theta$ to $I_{\beta}$, that send $0$ to $0$
  and $1$ to $1$, and make the diagram \eqref{eq:fiber product of
    intervals} commutative.
\end{proof}

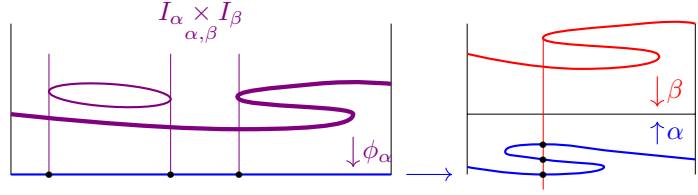
\begin{figure}
  \centering
  \begin{tikzpicture}[scale=.1]
    \tikzset{
      a/.style={red,thick},
      b/.style={blue,thick},
      ab/.style={violet,thick},
    }
    \begin{scope}[shift={(-30,0)}]
      \draw[very thin](0,20)--(0,0)--(50,0)--(50,20);
      \draw[ab,ultra thick] plot[smooth, tension=1] coordinates{
        (00,8)
        (30,6)
        (45,8)
        (30,10)
        (40,12)
        (50,12)
      };
      \draw[ab] plot[smooth cycle, tension=1] coordinates{
        (05,11)
        (13,9)
        (21,10)
        (13,12)
      };

      \node[ab] at(25,20){$I_{\alpha}\fibertimes{\alpha}{\beta}I_{\beta}$};
      \draw[b](0,0)--(50,0);
%      \draw[b,line width=2pt](05,0)--(21,0);
      \draw[ab,very thin](05,0)--+(0,16)(21,0)--+(0,16)(30,0)--+(0,16);
      \node at(05,0){\tiny$\bullet$};
      \node at(21,0){\tiny$\bullet$};
      \node at(30,0){\tiny$\bullet$};

      \draw[ab,thin,->](45,5)--(45,1)node[pos=.5,right]{$\phi_{\alpha}$};
    \end{scope}

    \draw[b,thin,->](22,0)--(28,0);
    
    \begin{scope}[shift={(30,0)}]
      \draw[very thin](0,20)--(0,0)(0,8)--(30,8)(30,0)--(30,20);
      \draw[a,shift={(0,3)}] plot[smooth, tension=1,shift={(0,5)}] coordinates{
        (00,8)
        (15,6)
        (25,8)
        (10,10)
        (23,12)
        (30,12)
      };
      \draw[b] plot[smooth, tension=1] coordinates{
        (00,1)
        (10,0)
        (18,1)
        (10,2)
        (05,3)
        (10,4)
        (20,3)
        (30,2)
      };

      \draw[a,very thin](10,-2)--+(0,20);
      \node at(10,0){\tiny$\bullet$};
      \node at(10,2){\tiny$\bullet$};
      \node at(10,4){\tiny$\bullet$};
      \draw[a,->,thin](25,13)--(25,9)node[pos=.5,right]{$\beta$};
      \draw[b,->,thin](25,4)--(25,7)node[pos=.5,right]{$\alpha$};
    \end{scope}
  \end{tikzpicture}
  \caption{One component of fiber product
    $I_{\alpha}\fibertimes{\alpha}{\beta}I_{\beta}$ 
    maps with degree $1$ on $I_{\alpha}$.}
  \label{fig:pullbacklemma}
\end{figure}

\begin{lemma}\label{lem:reparmetrization}
  Consider a (generic) loop $\gamma$ in $M$, the path
  $\Gamma:[0,1]\to\cM(\gamma,M)$ associated to the crocodile walk along
  $\gamma$, and let $\alpha=\pi_{1}\circ\Gamma$ (i.e. $\alpha$ is the map
  from $[0,1]$ to $[0,1]$ such that the trajectory $\Gamma(t)$ starts
  at the the point $\gamma(\alpha(t))$).

  Let $\beta:[0,1]\to[0,1]$ be a map such that
  $\alpha$ and $\beta$ satisfy the assumptions in
  lemma \ref{lem:fiberproductofintervals}.

  Then, with the notation of lemma \ref{lem:fiberproductofintervals},
  $\Gamma\circ\phi_{\alpha}$ still extends a sequence of consecutive
  steps, that is the same as the original one, up to cancellation of
  consecutive opposite steps.
%  
%  Consider a path $\Gamma:[0,1]\to\cM(\gamma,M)$ that extends a sequence
%  of consecutive Morse steps, and such that
%  $\alpha=\pi_{1}\circ\Gamma:[0,1]\to[0,1]$ sends $0$ to $0$ and $1$ to
%  $1$ (typically, a crocodile walk). Consider a  map $\phi:[0,1]\to[0,1]$
%  sending $0$ to $0$ and $1$ to $1$.
%
%  Then $\Gamma\circ\phi$ is homotopic to path $\Gamma_{\phi}$ such that
%  \begin{itemize}
%  \item
%    $\pi_{1}\circ\Gamma\circ\phi = \pi_{1}\circ
%    \widetilde{\Gamma_{\phi}}$
%  \item
%    $\Gamma_{\phi}$ extends a sequence of consecutive paths, that is the
%    same as the initial one up to cancellation of consecutive opposite
%    steps.
%  \end{itemize}
\end{lemma}

\begin{remark}
  For a general reparametrization $\phi$, notice that  $\Gamma\circ\phi$
  may fail to genuinely be an extension of a sequence of steps, as it may
  make a U-turn while in the middle of a (lower) step for instance. Nevertheless,
  $\pi_{1}\circ\Gamma$ is constant when $\Gamma$ is in a lower step, so
  that $\Gamma\circ\phi$ can be deformed to either avoid this step all at
  once or go through it injectively, while keeping the projection
  $\pi_{1}\circ\Gamma\circ\phi$ unchanged.

  The above statement then holds for general reparametrizations, up to
  these slight deformations.
\end{remark}

\begin{proof}
  To fix notations, let $0=t_{0}<\dots <t_{2N}=1$ be the times at which
  $\Gamma$ switches from a low/upper step to an upper/lower step of the
  crocodile walk: $\Gamma$ is going through a lower step on
  $[t_{2i},t_{2i+1}]$ and through an upper step on $[t_{2i+1},t_{2i+2}]$.

  Observe that $\alpha$ is constant on an interval $[t_{2i},t_{2i+1}]$ on
  which $\Gamma$ is going through a lower step. By construction of
  $\phi_{\alpha}$, this means that whenever $\phi_{\alpha}$ enters
  $[t_{2i},t_{2i+1}]$ through one end, it goes through it injectively up
  to the other end. In other words, whenever $\Gamma\circ\phi_{\alpha}$ enters
  a lower step, it is travelled injectively up to the other side:
  %
%  \textcolor{red}{(j'ai compris en regardant la tête du produit fibré, il
%    y a une manière de le voir plus naturellement ?)}
  %
  in particular, this
  implies that the path $\Gamma\circ\phi_{\alpha}$ is indeed an extension
  of a sequence of consecutive Morse steps.

  Moreover, $\phi_{\alpha}$ sends $0$ to $0$ and $1$ to $1$, so that it
  is homotopic to the identity: we can pick a generic homotopy
  $\{\phi_{s}\}$ for which, except for a finite number of values
  $\{s_{1},\dots,s_{k}\}$ of $s$, all the times $t_{i}$ are regular
  values of $\phi_{s}$, and at a bifurcation time $s_{i}$, exactly one
  critical value, associated to exactly one non degenerate critical point
  does cross a time $t_{j}$.

  Consider a parameter $s\notin\{s_{1},\dots,s_{k}\}$. Then each time
  $\phi_{s}$ enters an interval $[t_{2i},t_{2i+1}]$, it does either exit
  it through the same end or through the other one, and in both cases,
  $\phi_{s}$ can be deformed to either be constant at the relevant end or
  to travel through the interval injectively, without modifying the
  projection $\pi_{1}(\Gamma\circ\phi_{s})$. After this modification,
  $\Gamma\circ\phi_{s}$ is an extension of a sequence of consecutive
  steps.
% 
%  \textcolor{red}{Pourquoi opposite ? La modification fait juste de
%  $\phi_{s}$ une extension au sens de la définition 6.1. non ? }
%
  
  Moreover, this sequence only depends on the direction in which
  $\phi_{s}$ successively crosses the set $\{t_{0},\dots,t_{2N}\}$, and
  hence the sequence of Morse steps is locally constant away from the
  bifurcations $\{s_{1},\dots,s_{k}\}$.

  Finally, when crossing a bifurcation parameter $s_{i}$, the sequence of
  Morse steps is modified by either deleting or inserting a pair of
  opposite Morse steps. As a consequence, the initial and last sequences are
  the same up to cancellation of consecutive opposite Morse steps.
\end{proof}
 
\subsection{Homotopy lemmas}

\begin{lemma}\label{lem:homotopy reg down}
  Let $(\gamma_{s})_{s\in[0,1]}$ be a continuous family of based loops
  in $M$ such that for each $s$, $\gamma_{s}$ is transversal to all
  the evaluation maps $\cM(M,M)\xto{\ev_{-}}M$. Then
  $\Theta(\gamma_{0})=\Theta(\gamma_{1})$.
\end{lemma}
\begin{proof}
  Since all the relevant maps remain transversal, the moduli space
  $\cM(\gamma_{s},M)$ and its boundary depends continuously on $s$. Since
  the set of Morse steps is discrete, we conclude that
  $\Theta(\gamma_{s})$ is constant with respect to $s$.
\end{proof}

\begin{lemma}\label{lem:homotopy reg up}
  Let $(\gamma_{s})_{s\in[0,1]}$ be a continuous family of based loops in
  $M$ such that the two ends $\gamma_{0}$ and $\gamma_{1}$ are
  transversal to $\cM(M,M)\xto{\ev_{-}}M$. Suppose moreover that for each
  $s$, $\gamma_{s}$ is transversal to the evaluation map
  $\cM(M,M)\xto{\ev_{+}}M$. Then $\Theta(\gamma_{0})=\Theta(\gamma_{1})$
  up to relations.
\end{lemma}

\begin{proof}
  Observe that our assumption essentially means that the moduli spaces
  defining the upward crocodile walk on $\gamma_{s}$ are always well cut
  out transversely, and hence that the upward crocodile walk is constant.
  However, the loops $\gamma_{s}$ are based at $\star$, and to avoid
  transversality/combinatorial discussions, we want to use the upward
  crocodile walk on loops based at $\ostar$ (see discussion before
  definition \ref{def:costep}).

  To this end, consider the path $\gamma_{s}'=\delta^{-1}\gamma_{s}\delta$,
  where $\delta$ is a path from $\star$ to $\ostar$ along which no
  transversality issue occurs.

  Recall from section \ref{sec:crocodilewalks} that the crocodile walk along
  $\gamma_{\epsilon}$ ($\epsilon\in\{0,1\}$) gives rise to a sequence of paths associated to
  alternating upper and lower gluings
  $$
  \begin{tikzcd}[start anchor=east,end anchor=west]
    \llap{$\{\star\}=\,$}
    u_{\epsilon,1}\ar[r,bend right=20,"\sigma_{\epsilon,1}",pos=.7]&
    u_{\epsilon,2}\ar[r,bend left =20,"h_{\epsilon,1}"]&
    u_{\epsilon,3}\ar[r,bend right=20,"\sigma_{\epsilon,2}"]&
    u_{\epsilon,4}\ar[r,phantom,"\dots"]&
    u_{\epsilon,2l-1}\ar[r,bend right =20,"\sigma_{\epsilon,l}",pos=.8]&
    u_{\epsilon,2l}\rlap{$\,=\{\star\}$}
  \end{tikzcd}
  $$

  Similarly, the upward crocodile walk along $\gamma'_{s}$ gives rise to a
  sequence
  $$
  \begin{tikzcd}[start anchor=east,end anchor=west]
    \llap{$\{\ostar\}=\,$}
    v_{s,1}\ar[r,bend left=20,"\tau_{1}",pos=.5]&
    v_{s,2}\ar[r,bend right =20,"g_{s,1}"]&
    v_{s,3}\ar[r,bend left=20,"\tau_{2}"]&
    v_{s,4}\ar[r,phantom,"\dots"]&
    v_{s,2k-1}\ar[r,bend left =20,"\tau_{k}",pos=.5]&
    v_{s,2k}\rlap{$\,=\{\ostar\}$}
  \end{tikzcd}
  $$
  
  Notice that our transversality assumption ensures that the steps
  $\tau_{1},\dots,\tau_{k}$ do not depend on $s$, and that we can pick
  parametrizations of the whole path such that the upper steps are
  synchronous for all $s$.

  The fact that the downward and upward walks above do not start and end
  at the same points is intentional to avoid artificial transversality
  issues. However, for them to still evaluate on the same topological
  loop in $M$, we now introduce a slight modification of the
  downward walks.
  
  We first fix some notation: the step $\sigma_{\epsilon,1}$ goes from
  the zero length trajectory $u_{\epsilon,1}=(\star,\star,\star)$ to
  $$
  u_{\epsilon,2} = (\star,\beta'_{\epsilon,1},\alpha_{\epsilon,2})
  \in \{\star\}\times\cM(\star,x_{\epsilon,1})\times\cM(x_{\epsilon,1},M),
  $$
  and $h_{\epsilon,1}$ from
  $u_{\epsilon,2}=(\star,\beta'_{\epsilon,1},\alpha_{\epsilon,2})$ to
  $u_{\epsilon,3}=(\gamma_{\epsilon,2},\beta_{\epsilon,2},\alpha_{\epsilon,2})$.
  
  Since no transversality issue is supposed to occur along $\delta$,
  the trajectory $\beta'_{\epsilon,1}$ can be tracked when $\star$ is
  moved from $\star$ to $\ostar$ along $\delta$: this defines a
  $1$-parameter family of trajectories in $\cM(\delta,x_{\epsilon,1})$,
  running from $\beta'_{\epsilon,1}$ to a trajectory
  $\tilde{\beta'}_{\epsilon,1}\in\cM(\ostar,x_{\epsilon,1})$.
  
  Concatenating these trajectories with $\alpha_{\epsilon,2}$, and
  reversing the orientation, we obtain a path $\tilde{h}_{\epsilon,1}$ in
  $\cM(\delta,M)$ from
  $\tilde{u}_{\epsilon,2}=(\ostar,\tilde{\beta'}_{\epsilon,1},\alpha_{\epsilon,2})$ to
  $u_{\epsilon,2}=(\star,\beta'_{\epsilon,1},\alpha_{\epsilon,2})$.

  Doing the same thing at the end of the walk, we finally obtain the
  following modified sequences
  $$
  \begin{tikzcd}[start anchor=east,end anchor=west]
    \tilde{u}_{\epsilon,2}\ar[r,bend left=20,
        "\tilde{h}_{\epsilon,1}"]&
    u_{\epsilon,2}\ar[r,bend left =20,"h_{\epsilon,1}"]&
    u_{\epsilon,3}\ar[r,bend right=20,"\sigma_{\epsilon,2}"]&
    u_{\epsilon,4}\ar[r,phantom,"\dots"]&
    u_{\epsilon,2l-1}\ar[r,bend left =20,
    "h_{\epsilon,l-1}\cdot \tilde{h}_{\epsilon,l-1}"]&
    \tilde{u}_{\epsilon,2l-1}
  \end{tikzcd}
  $$
  where
  $$
    \tilde{u}_{\epsilon,2}=(\ostar,\tilde{\beta'}_{\epsilon,1},\alpha_{\epsilon,2})
    \qquad\text{ and } \qquad
    \tilde{u}_{\epsilon,2l-1}=(\ostar,\tilde{\beta}_{\epsilon,2l-1},\alpha_{2l-2}).
  $$

  The net outcome of this discussion is that, up to a minor modification
  of the first and last steps of the original downward crocodile walks
  (i.e. associated to the base point $\star$), we obtain upward and
  downward walks along the same path $\gamma'_{\epsilon}$
  ($\epsilon\in\{0,1\}$), that define paths
  $$
  I^{\downarrow}_{\epsilon}=[0,1]\xto{\CrocoDo_{\epsilon}}\cM(\gamma'_{\epsilon},M)
  \quad\text{ and }\quad
  I^{\uparrow}_{\epsilon}  =[0,1]\xto{\CrocoUp_{\epsilon}}\cM(M,\gamma'_{\epsilon})
  $$
  and for a generic choice of data, the composed maps
  $$
  \begin{tikzcd}
    I^{\downarrow}_{\epsilon}
    \arrow[r,"\CrocoDo_{\epsilon}"]
    \arrow[rr,bend left=30, "\alpha_{\epsilon}"]
    &
    \cM(\gamma'_{\epsilon},M)\arrow[r,"\pi_{1}"]
    &
    {[0,1]}
  \end{tikzcd}
  \quad\text{ and }\quad
%  I^{\uparrow}  =[0,1]\xto{\CrocoUp}\cM(M,\gamma')\xto{\pi_{2}}[0,1]
   \begin{tikzcd}
     I^{\downarrow}_{\epsilon}
     \arrow[r,"\CrocoUp_{\epsilon}"]
     \arrow[rr,bend left=30, "\beta_{\epsilon}"]
    &
    \cM(M,\gamma'_{\epsilon})\arrow[r,"\pi_{2}"]
    &
    {[0,1]}
  \end{tikzcd}
  $$
  (where $\pi_{i}$ is the corresponding projection in the fiber product
  to the source of $\gamma'_{\epsilon}$) have no common singular values.

  Applying lemma \ref{lem:fiberproductofintervals}, we obtain maps
  $\phi_{\epsilon,\alpha}$ and $\phi_{\epsilon,\beta}$ sending $0$ to $0$
  and $1$ to $1$ making the following diagram commutative: 
  $$
  \begin{tikzcd}
    {[0,1]}\arrow[rr,"\phi_{\epsilon,\beta}"]\arrow[dd,"\phi_{\epsilon,\alpha}"']
    &&
    I^{\uparrow}_{\epsilon}\arrow[d,"\CrocoUp_{\epsilon}"]\\
    &&\cM(M,\gamma'_{\epsilon,})\arrow[d,"\pi_{2}"]\\
    I^{\downarrow}_{\epsilon}\arrow[r,"\CrocoDo_{\epsilon}"]
    &
    \cM(\gamma'_{\epsilon},M)\arrow[r,"\pi_{1}"]
    &
    {[0,1]}
  \end{tikzcd}
  $$
  
  Moreover, according to lemma \ref{lem:reparmetrization}, the lower part
  of $\Gamma^{\downarrow}_{\epsilon}\circ\phi_{\epsilon,\alpha}$ is the
  same as the lower part of  $\Gamma^{\downarrow}_{\epsilon}$ up to
  cancellation of consecutive opposite steps: in particular, they are
  the same in $\MSLoops(\Xi,\star)/\cR$ and we can replace $\CrocoDo_{\epsilon}$
  with $\Gamma_{\epsilon}^{\downarrow}\circ\phi_{\epsilon,\alpha}$.

  Replacing also $\CrocoUp_{\epsilon}$ with
  $\CrocoUp_{\epsilon}\circ\phi_{\epsilon,\beta}$, we end up as before
  with two upward and downward walks on $\gamma'_{\epsilon}$, with the
  additional property that they are now synchronized, i.e.:
  $$
  \forall t\in[0,1], \pi_{1}\circ\CrocoDo_{\epsilon}(t)
  = \pi_{2}\circ\CrocoUp_{\epsilon}(t).
  $$

  This means that at every time $t\in[0,1]$, the configuration
  $(\CrocoUp_{\epsilon}(t),\CrocoDo_{\epsilon}(t))$ can be seen as a
  bouncing trajectory: this defines a path $\Gamma^{\dag}_{\epsilon}$
  in $\bM(M,M)$.

  More precisely, consider an interval $[t_{2i},t_{2i+1}]$ on which
  $\CrocoUp_{\epsilon}$ is a lower step $g_{\epsilon,i}$. All the
  trajectories $\CrocoUp_{\epsilon}(t)$ in this interval are of the form
  $$
  g_{\epsilon,i}(t)=(\omega_{i},\xi_{\epsilon,i}(t))
  $$
  for a fixed trajectory $\omega_{i}\in\cM(M,a_{i})$ for some  index $n$
  critical point $a_{i}$, and where $\xi_{\epsilon,i}$ is path in
  $\cM(a_{i},M)$ from a boundary trajectory $(\phi'_{i},\chi_{i})$ to
  another one $(\phi_{i+1},\chi_{i+1})$. Here either
  $\phi'_{i}\in\cM(a_{i},b_{i})$ for some index $n-1$ critical point and
  $\chi_{\epsilon,i}\in\cM(b_{i},\gamma_{\epsilon})$, or
  $\phi'_{i}\in\cM(a_{i},\ostar)$, and $\chi_{\epsilon,i}$ is the
  constant trajectory at $\ostar$ (see figure \ref{fig:crocopatch}).

  \begin{figure}
    \centering
    \begin{tikzpicture}[z={(-.866,-.5)}]
      \tikzset{txt/.style={fill=white}}

      \coordinate(c00)at(-4,0, 1);
      \coordinate(c01)at( 4,0, 1);
      \coordinate(c10)at(-4,0,-1);
      \coordinate(c11)at( 4,0,-1);
      \coordinate(b0) at(-4,3,0);
      \coordinate(b1) at( 4,3,0);
      \coordinate(a)  at( 0,5,0);
      \coordinate(m0) at(-5,7,0);
      \coordinate(m1) at( 5,7,0);
      
      \draw[noEnd-noEnd,thick](-5,0,-1)--(5,0,-1);
      \begin{scope}
        \foreach \ds in  {-.8,-.6,...,.8}{
          \draw[noEnd-noEnd,very thin](-5,0,\ds)--(5,0,\ds);
        }

%        \draw[densely dashed](c00)\myto{AA}{out=rand*20,in=180+rand*20,relative}(c10);
        \draw[trajectory,dashed](b0)to[out control=+(  20:.75),in=90](c10);
%        \draw[trajectory,dashed](b0)to[out control=+(-100:.75),in=90](c10);

        \fill[white,opacity=.75]
        (a)to[out=-180,in=0]
        (b0) to[out control=+(-160:.75),in=90]
        (c00)--
        (c01)to[out control=+(90:0),in control=+(-160:.75)]
        (b1)to[out=180,in=-  0](a);

        \draw[noEnd-noEnd](m0)--(m1);
        \draw[trajectory]([shift={(0,2)}]a)--(a)node[pos=.5,left]{$\omega_{i}$};
        \draw[trajectory](a)to[out=-180,in=0]
        node[pos=.5,above left]{$\phi'_{i}$}(b0);
        \draw[trajectory](a)to[out=-  0,in=180]
        node[pos=.5,above right]{$\phi_{i+1}$}(b1);
        \draw[trajectory       ](b0)to[out control=+(-160:.75),in=90]
        node[pos=.5,left]{$\chi_{0,i}$}(c00);
        \draw[trajectory       ](b1)to[out control=+(-160:.75),in=90]
        node[pos=.5,right]{$\chi_{0,i+1}$}(c01);
        \draw[trajectory       ](b1)to[out control=+(  20:.75),in=90](c11);
%        \draw[trajectory       ](b1)to[out control=+(-100:.75),in=90](c11);

        \path(b0)\myto{AA}{out=-15,in=180}(0,2.5,.5)\myto{AA1}{out=0,in=-170}
        ([shift={(-.2,0)}]b1);
        \path(c00)\myto{BB}{}(0,0,1)\myto{BB1}{}(c01);

        \foreach \i in {1,2,...,9}{
          \draw[othertraj]%,rounded corners=10pt]
          (a)   to[out control=+(-180+\i/10*70:1.5),in control=+(10+\i/10*80:1.5)]
          (AA\i)to[out control=+( 195+\i/10*80:.5),in control=+(90:1)]
          (BB\i);
        }
        \foreach \i in {0,1,2,...,9}{
          \draw[othertraj,rounded corners=3pt]
          (a)to[out control=+(-110+\i/10*110:2-\i/10*.5),
          in control=+(80+\i/10*80:1.5+\i/10*1.5)]
          (AA1\i)to[out control=+(-100+\i/10*20:1.5-\i/10*1.5),
          in control=+(90:1+\i/10*2)]
          (BB1\i);
        }

        \path(m0)\myto{BB}{}(0,7,0);
        \path(m1)\myto{BB1}{}(0,7,0);
        \path([shift={(0,-2)}]m0)\myto{AA}{}([shift={(-.3,.3)}]a);
        \path([shift={(0,-2)}]m1)\myto{AA1}{}([shift={(.3,.3)}]a);
 
        \foreach \i in {1,2,...,9}{
          \draw[othertraj]
          (b0)   to[out control=+(120-\i/10*120:1.5),in control=+(-80-\i/10*80:1)]
          (AA\i)to[out control=+(100-\i/10*80:.5),in control=+(-90:1)]
          (BB\i);
        }
        \foreach \i in {1,2,...,9}{
          \draw[othertraj]
          (b1)   to[out control=+(60+\i/10*120:1.5),in control=+(-100+\i/10*80:1)]
          (AA1\i)to[out control=+(80+\i/10*80:.5),in control=+(-90:1)]
          (BB1\i);
        }

%        \draw[densely dashed](c01)\myto{AA}{out=rand*20,in=180+rand*20,relative}(c11);
%        \foreach \i in {1,2,...,9}{
%          \draw[othertraj,thin]
%          (b1)   to[out control=+(-160+\i/10*60:.75),in control=+(-80-\i/10*80:0)]
%          (AA\i);
%        }

        \path (a) node{\tiny$\bullet$} node[above left]{$a_{i}$};
        \path (b0)node{\tiny$\bullet$} node[above left]{$b_{i}$};
        \path (b1)node{\tiny$\bullet$} node[above right]{$b_{i+1}$};

        \draw[noEnd->,thick](-4,5.5)to[out=0,in=150]node[pos=.1,above]{$\tau_{i}$}(-2,5);
        \draw[-{noEnd}>,thick](2,5)to[out=30,in=180]node[pos=.9,above]{$\tau_{i+1}$}(4,5.5);
        \draw[->,thick](-2,3.5)to[out=-20,in=-160]
        node[pos=.5,above]{\tikzoutline{$g_{0,i}$}}
        node[pos=.5,below]{\tikzoutline{${\Gamma_{0}^{\uparrow}}([t_{2i},t_{2i+1}])$}}
        (2,3.5);

        \begin{scope}[
          decoration={
            segment length=45pt,
            amplitude=25pt},
          croco/start=continue,
          croco/end=midup,
          ]
          \clip($(c00)+(-1,-2.5)$)rectangle($(c01)+(1,.5)$);
          \draw[croco/start=midup,croco]($(c00)+(-1.5,0)$)--(c00);
          \draw[croco](c00)--(c01);
          \draw[croco](c01)--($(c01)+(1.5,0)$);
        \end{scope}
        \draw[noEnd-noEnd,thick](-5,0, 1)coordinate(deb)--(5,0, 1)coordinate(fin);

        \draw(5,0,1)node[right]{$\gamma_{0}$};
        \draw(5,0,-1)node[right]{$\gamma_{1}$};
        \draw(c00)node{\tiny$\bullet$};
%        node[below]{\tikzoutline{$\gamma_{0}(\alpha_{0}^{\uparrow}(t_{2i}))$}};
        \draw(c01)node{\tiny$\bullet$};
%        node[below]{\tikzoutline{$\gamma_{0}(\alpha_{0}^{\uparrow}(t_{2i+1}))$}};
%       
%        \draw(c01)node{\tiny$\bullet$}node[below]{$\gamma_{0}(\pi_{2}(t_{2i+1}))$};
%        \draw(c00)node{\tiny$\bullet$}node[below]{$\gamma_{0}(\pi_{2}(t_{2i}))$};

      \end{scope}
      
    \end{tikzpicture}    
    \caption{Upward and downward crocodile walks define matching patches.
    }
    \label{fig:crocopatch}
  \end{figure}
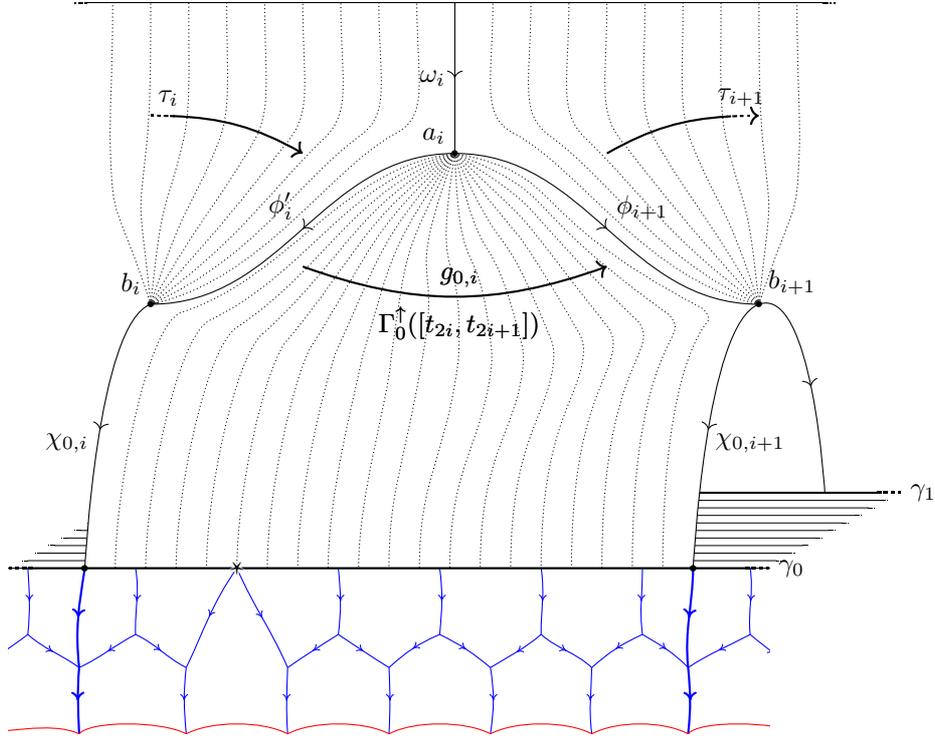

%\end{document}

  Notice that neither $a_{i}$, nor $\phi'_{i}$ nor $\phi_{i+1}$ depends
  on $\epsilon$. As a consequence, the paths $\lambda_{0} =
  (\xi_{0,i},{\CrocoDo_{0}}_{|_{[t_{2i},t_{2i+1}]}})$ and $\lambda_{1} =
   (\xi_{1,i},{\CrocoDo_{1}}_{|_{[t_{2i},t_{2i+1}]}})^{-1}$, seen as
  elements of $\bM(a_{i},M)$ form a relation patch.

  Such a patch has exactly $2$ open edges, associated to $\phi'_{i}$ and $\phi_{i+1}$.

  Consider now the first patch, associated to $a_{1}$, and in this patch,
  the ``first'' open edge, i.e. the open edge associated  $\phi'_{1}$.

  It is associated to the start of the upward walk and the start of the
  modified downward walk:
%  \textcolor{red}{(qui sont aussi par définition
%    du crocodile la fin des chemins respectifs, c'est bête mais je
%    n'avais pas compris tout de suite)}
  the part before the bounce reduces to the $0$ length trajectory at
  $\ostar$ and the part after the bounce is the trajectory
  $\tilde{u}_{\epsilon,2}$.
%  \textcolor{red}{(oui, car on parle de l'open
%    edge et pas de la trajectoire de $\bM(a_1,M)$)}.
  It consists of the
  two trajectories
  $$
  (\ostar, \tilde{u}_{0,2}) = (\ostar, \tilde{\beta'}_{0,1},
  \alpha_{0,2}))
  \qquad \text{ and }\qquad
  (\ostar, \tilde{u}_{1,2}) = (\ostar, \tilde{\beta'}_{1,1},
  \alpha_{1,2}))
  $$

  Recall that  $(\beta_{0,1},\alpha_{0,2})$ is one end of the step
  $\sigma_{0,1}$ rooted at $\star$ and whose other end is the zero length
  trajectory at $\star$. The same holds for $(\beta_{1,1},\alpha_{1,2})$,
  and in particular, we have
  $(\beta_{0,1},\alpha_{0,2})=(\beta_{1,1},\alpha_{1,2})$.

  This means that the two trajectories forming the open edge are in fact
  the same: this open edge is artificial and can be considered as
  non-existent.

  \medskip

  The same holds for the ``last'' open edge of the last patch. This edge
  is artificial and can be removed, so that in general, the last patch
  has only one open edge, unless there is only one patch, in which case
  the first and last ones are the same, and it has no open edge at all.
  %
%  \textcolor{red}{(je ne suis pas sûr
%    de comprendre la phrase : on dit que chaque nouveau patch a soit une
%    open edge en commun avec la précédente, soit zéro et alors on a
%    terminé ?)}
  %
  This means that our
  sequence of patches is a line of matching relations patches, according
  to definition \ref{def:tree of matching patches}.

\begin{figure}
  \centering
  \begin{tikzpicture}[%z={(-.866,-.5)},
    remember points/count=11,
    remember points/prefix=PP,
    decoration={amplitude=30pt,segment length=80pt},%
    croco/start=continue,%
    croco/end=midup,%
    croco/save nodes=true,%
    croco/save node prefix=AA,%
    croco/proba=.75,%
    ]
    \pgfmathsetseed{4}
    \begin{scope}
      \clip(-5.75,4)rectangle(3.2,-3.6);

      \coordinate(a) at (-1.5,3);
      \coordinate(m) at ($(a)+(0,3)$);
      \begin{pgfonlayer}{foreground}
        \path (a) node{\tiny$\bullet$}node[above]{\tikzoutline{$a$}};
      \end{pgfonlayer}
      \draw[croco/start=midup,croco,remember points]
      (-6,1)to[out=-20,in=180](0,0);
      \coordinate(D)at(ctD);
      \coordinate(C)at(ctC);
      \draw[croco,croco/verticality=.75,remember points/prefix=QQ,remember points]
      (0,0)to[out=15,in=-165](3.5,.75);
      
      \foreach \i in {1,...,9}{
        \draw[othertraj](a)to[out=-180+\i*15,in=70+\i*3](PP\i);
      }
      \foreach \i in {0,...,2}{
        \draw[othertraj](a)to[out=-180-\i*10,in=60-\i*3]($(PP1)+(-.5,\i/2+.5)$);
      }
      \foreach \i in {2,4,...,8}{
        \draw[othertraj](a)to[out=-30+\i*3,in=100-\i](QQ\i);
      }
      \foreach \i in {0,...,14}{
        \draw[othertraj](a)to[out=\i*10,in=180+\i*10-20]++(\i*10-10+10*\i/14:4);
      }

      \coordinate(s)at($(C)+(60:1.7)$);
      %% Star step
      \fill[white, opacity=.75]
        (0,0)--
        (s)to[out=-60,in=0]
        (D)to[relative, out=-10,in=180+5]
        ([xshift=2pt] C)to[relative, out=-10,in=180+5](0,0)--cycle;
      \draw[remember points,remember points/prefix=SS](0,0)--(s);
%      \foreach\i in{2,4,...,8}{
%        \draw[othertraj](SS\i)--(C);
%      }
      \draw[red,remember points](s)to[out=-60,in=0](D);
      \draw[thick,->-,blue](s)to[relative, out=-5,in=180-5](C);
      \foreach \i in {1,...,9}{
        \draw[othertraj](s)to[out=-120,in=180-9*\i](PP\i);
      }
        
      %% OStar step
%      \fill[white, opacity=.75]
%        (0,0)to[out=-30,in=0]
%        (D)to[relative, out=-10,in=180+5]
%        ([xshift=.75pt] C)to[relative, out=-10,in=180+10](0,0)--cycle;
%      \draw[red,remember points,save path=\toto](0,0)to[out=-30,in=0](D);
%      \foreach \i in {1,...,9}{
%        \draw[othertraj](0,0)to[bend right=15-\i](PP\i);
%      }

      \begin{pgfonlayer}{foreground}
        \node at (0,0){$\wostar$};
        \node at (s){$\wstar$};
        \node[above] at ($(s)!.5!(0,0)$){\tikzoutline{$\delta$}};
      \end{pgfonlayer}

      \draw[trajectory] (a)to[out=-30,in=100](0,0);
      \draw[-noEnd] (a)to[out=150,in=-30]+(-2,1);

      \draw[xshift=-5pt,<-,red](0.5,-3.5)to[out=140,in=-10](-3,-2.75);
      \draw[xshift= 5pt,->,red](0.5,-3.5)to[out=140,in=-170](2.5,-2.5);
    \end{scope}
    
  \end{tikzpicture}
  \caption{Rooting the Morse loop at $\star$.}
  \label{fig:backtostar}
\end{figure}
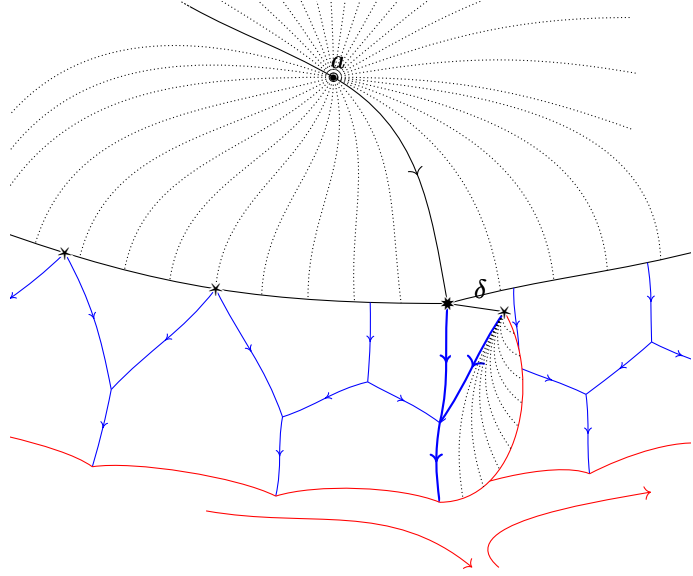

Let $\rho$ denote the (free) Morse loop forming the bottom part of our
matching patches. It starts with $\sigma_{0,2}$ and ends with
$\sigma_{1,2}^{-1}$.

It only remains to observe that the step $\sigma_{0,1}=\sigma_{1,1}$ runs
from the zero length trajectory $\star$ to
$(\beta'_{0,1},\alpha_{0,2})=(\beta'_{1,1},\alpha_{1,2})$. As a
consequence, $\rho$ can be conjugated by $\sigma_{0,1}$ (see figure \ref{fig:backtostar}):
$\sigma_{0,1}\cdot\rho\cdot\sigma_{0,1}^{-1}$ is a relation, and it is
also equal to $\Theta(\gamma_{0})\Theta(\gamma_{1})^{-1}$ up to
cancellation of consecutive opposite steps.
\end{proof}

\subsection{Proof of theorem \ref{thm:pi1}}

\begin{proof}[Proof of theorem \ref{thm:pi1}]
Consider a Morse loop $\sigma=\sigma_{1}\cdots\sigma_{k}$ and suppose
that its evaluation $\gamma=\ev(\sigma)$ in $M$ is trivial.

Since $\sigma=\Theta(\sigma)$ up to a relation, we can replace $\sigma$
by $\Theta(\sigma)$. We now have a topological loop $\gamma$ in $M$ that
is homotopically trivial, and such that the downward crocodile along
$\gamma$ is $\sigma$.

Consider an homotopy $\{\gamma_{s}\}_{s\in[0,1]}$ from $\gamma$ to the
constant loop at the base point. For a generic choice of this homotopy,
the loop $\gamma_{s}$ is transverse to all the relevant evaluation maps
required to define the downward crocodile walk, except at a finite number
of isolated values $0<s_{1}<\dots<s_{N}<1$. In the same way, $\gamma_{s}$
can be supposed to be transverse to all the relevant evaluation maps
required to define the upward crocodile walk, except for a finite set of
values $0<s'_{1}<\dots<s'_{N'}<1$. Finally, for a generic choice of the
homotopy, these two sets of values can be supposed to be disjoint:
\begin{equation}
  \label{eq:disjointcriticalsets}
  \{s_{i}\}_{1\leq i\leq N}\cap \{s'_{i}\}_{1\leq i\leq N'}= \varnothing.
\end{equation}

From lemma \ref{lem:homotopy reg down}, on the components of
$[0,1]\setminus\{s_{i}\}_{1\leq i\leq N}$, $\Theta(\gamma_{s})$ is
constant, and from lemma \ref{lem:homotopy reg up}, on the components of
$[0,1]\setminus\{s'_{i}\}_{1\leq i\leq N'}$, $\Theta(\gamma_{s})$ is
invariant up to relations. As a consequence, up to relations, we have
$\sigma=\Theta(\star)$, so that $\sigma$ is trivial up to relations.
  
\end{proof}

\section{Functoriality}

Consider now two closed finite dimensional manifolds $M$ and $N$ and a
continuous map $M\xto{\phi}N$.

Consider two sets of auxiliary data $\Xi_{M}=(f,X,\star)$ on $M$ and
$\Xi_{N}=(f',X',\star')$ on $N$, and a smooth perturbation $\tilde{\phi}$
of $\phi$ (still sending $\star$ to $\star'$). Define then the moduli
spaces of hybrid trajectories from a critical point $x$ of $f$ to a
critical point $x'$ of $f'$ as:
\begin{align*}
  \bMhyb[\phi](x,x')
  &= \cM(x,M)\fibertimes{\phi\circ\ev_{+}}{\ev_{-}}\cM(N,x')\\
  &= \{(u,v)\in\cM(x,M)\times\cM(N,x'), \phi(\ev_{+}(u))=\ev_{-}(v)\}.
\end{align*}
More generally, we call hybrid trajectories the elements of the fiber
product $\bMhyb[\phi](\cdot,\cdot)$ over $\phi\circ\ev_{+}$ and $\ev_{-}$
of all the moduli spaces of the form $\cM(\cdot,M)$ and $\cM(N,\cdot)$,
for all possible constraints at start and end.

\begin{figure}
  \centering
  \begin{tikzpicture}[
    relative, out=rand*10, in=180-rand*10,
    decoration={segment length=40pt, amplitude=20pt},
    wtrajectory/.style={
      white,
      line width=2.5pt,
      postaction={draw,#1,->-}}
    ]
    \draw[->](-.5,0) ..controls+(1,.5)and+(-1,.5)..node[pos=.5,above]{$\phi$} (4.5,0);
    \begin{scope}[shift={(-3,0)},x={(-.5cm,-.3cm)},y={(1cm,-.1cm)},z={(0cm,1cm)}]
      \draw[white,line width=4pt](0,4,0)--(0,4,3);
      \draw[thin](0,0,0)--(0,0,3)--(0,4,3)--(0,4,0)--cycle;
      \draw[thin](0,0,0)--(3,0,0)--(3,4,0)--(0,4,0)--cycle;
      \coordinate(o) at(0,0,0);
      \coordinate(oo)at(0,4,0);
      \coordinate(y) at(1.5,2,2.5);
      \coordinate(x1)at(1.5,1,1.5);
      \coordinate(x2)at(1.5,3,1.5);
      \coordinate(m1)at(1.5,1,0);
      \coordinate(m2)at(1.5,3,0);
      \coordinate(p1)at(0,1,0);
      \coordinate(p2)at(0,3,0);

      \draw[ultra thick](p1)--(p2);
      \draw[wtrajectory={blue,thick}](y) to(x1);
      \draw[wtrajectory={blue,thick}](x1)to(m1);
      \draw[wtrajectory={blue,thick}](y) to(x2);
      \draw[wtrajectory={blue,thick}](x2)to(m2);
      \draw(m1)to[out=20,in=180+10](m2);
      \draw[dashed](m1)--(p1);
      \draw[dashed](m2)--(p2);

      \foreach \i in{y,x1,x2}{
        \node at (\i){\tiny$\bullet$};
      }

      \node[below right] at (0,4,0){$M$};

    \end{scope}

    \begin{scope}[shift={(3,-.4)},x={(.5cm,-.3cm)},y={(1cm,.1cm)},z={(0cm,1cm)}]
      \draw[white,line width=4pt](0,0,0)--(0,0,3)(0,1,0)--(0,1,3);
      \draw[thin](0,0,0)--(3,0,0)--(3,4,0)--(0,4,0)--cycle;
      \draw[thin](0,0,0)--(0,0,3)--(0,4,3)--(0,4,0)--cycle;
      \coordinate(o) at(0,0,0);
      \coordinate(oo)at(0,4,0);
      \coordinate(y1)at(1,1.5,2);
      \coordinate(y2)at(1,2.5,2); 
      \coordinate(x1)at(1.5,1,1.25);
      \coordinate(x2)at(1.5,2,1.25);
      \coordinate(x3)at(1.5,3,1.25);
      \coordinate(m1)at(0,1,2.5);
      \coordinate(m2)at(0,3,2.5);
      \coordinate(p1)at(0,1,0);
      \coordinate(p2)at(0,3,0);
      \coordinate(n1)at(1.5,1,0);
      \coordinate(n2)at(1.5,2,0);
      \coordinate(n3)at(1.5,3,0);
%      \draw[blue,thick,trajectory](y) to(x1);
%      \draw[blue,thick,trajectory](x1)to(m1);
%      \draw[blue,thick,trajectory](y) to(x2);
%      \draw[blue,thick,trajectory](x2)to(m2);
      \draw[ultra thick](p1)--(p2);
      \draw[dashed](p1)--(m1);
      \draw[dashed](p2)--(m2);
      \draw[remember points, remember points/prefix=A,remember points/count=4,]
        (m1)to[out=20,in=180+10](m2);
      \draw[trajectory,red,thin](A0)to(x1);
      \draw[trajectory,red,thin](A1)to(y1);
      \draw[trajectory,red,thin](A2)to(y2);
      \draw[trajectory,red,thin](A3)to(x3);

      \draw[wtrajectory={red,thick}](y1)to(x1);
      \draw[wtrajectory={red,thick}](y1)to(x2);
      \draw[wtrajectory={red,thick}](y2)to(x2);
      \draw[wtrajectory={red,thick}](y2)to(x3);
      \draw[wtrajectory={red,thick}](x1)to(n1);
      \draw[wtrajectory={red,thick}](x2)to(n2);
      \draw[wtrajectory={red,thick}](x3)to(n3);

      \draw(n1)to(n2);
      \draw(n2)to(n3);

      \node[below left] at (0,0,0){$N$};
      
      \foreach \i in{y1,y2,x1,x2,x3}{
        \node at (\i){\tiny$\bullet$};
      }

    \end{scope}
    \end{tikzpicture}
  \caption{Direct map.}
  \label{fig:pushstep}
\end{figure}
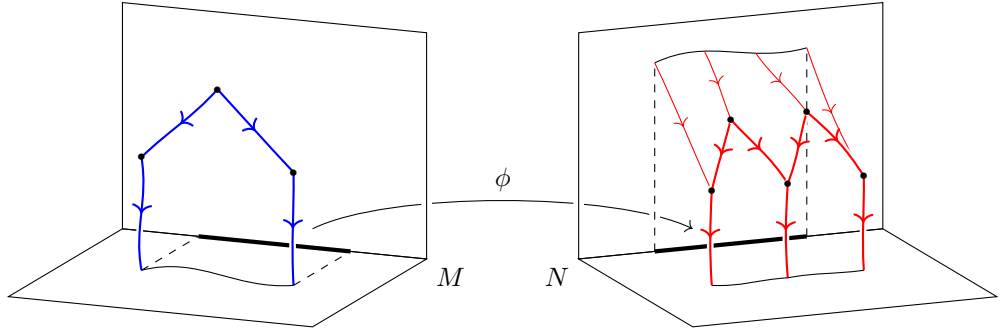

We work under the transversality assumption that all these moduli spaces
are cutout transversely, and from classical transversality arguments,
this is a generic condition on $\Xi$, $\Xi'$, and $\tilde{\phi}$.

The composition
$$
\Theta_{\tilde\phi}=\Theta_{N,\Xi'}\circ\tilde{\phi}\circ\ev_{+}
$$
of the evaluation in $M$, $\tilde{\phi}$ and the crocodile walk in $N$, turns a
Morse loop in $M$ into a Morse loop in $N$, and can be entirely described
in terms of moduli spaces of hybrid trajectories.

\begin{theorem}
  In the above situation, the map $\Theta_{\tilde{\phi}}$ induces a group
  morphism $\phi_{*}:\pi_{1}(M,\Xi)\to \pi_{1}(N,\Xi')$, that does not
  depend on the choice of perturbation $\tilde{\phi}$.

  When $M=N$ and $\phi=\Id$, $\phi_{*}$ is an isomorphism, and we
  identify the groups associated to different Morse data through this
  isomorphism.

  Then we have the following properties:
  \begin{enumerate}
  \item 
    $(\phi\circ\psi)_{*}=\phi_{*}\circ\psi_{*}$,
  \item 
    $(\Id)_{*}=\Id$,
  \item 
    if $\phi$ and $\psi$ are homotopic as based maps, then
    $\phi_{*}=\psi_{*}$.
  \end{enumerate}
\end{theorem}
\begin{proof}
  The strategy of proof of the theorem \ref{thm:pi1} (i.e. using
  upward crocodile walks to cross transversality issues affecting the
  downward crocodile walk) could be adapted in this new setting to derive
  the required properties. However, they are all a consequence of the
  following commutative diagram, and similar properties for the
  topological direct map:
  $$
  \begin{tikzcd}
    \pi_{1}(M,\star;\Xi )\ar[r,"\tilde{\phi}_{*}"]
    \ar[d,"\ev_{*}"',"\sim"{rotate=90,anchor=north}] &
    \pi_{1}(N,\star';\Xi')
    \ar[d,"\ev_{*}"',"\sim"{rotate=90,anchor=north}]\\
    \pi^{(top)}_{1}(M,\star)\ar[r,"\phi^{(top)}_{*}"]    &\pi^{(top)}_{1}(N,\star')\\
  \end{tikzcd}
  $$
  Here, commutativity of this diagram holds because crocodile walks
  preserve the homotopy class: given a Morse loop $\gamma$ in $M$,
  $\ev_{+}(\Theta_{N,\Xi'}(\tilde{\phi}(\ev_{+}(\gamma))))$ is homotopic
  to $\tilde{\phi}(\ev_{+}(\gamma))$, which in turn is homotopic to
  $\phi(\ev_{+}(\gamma))$.
\end{proof}

\subsection{Moving the base point}

The previous section provided a morphism associated to a change of
auxiliary data. To complete the picture, we now describe what happens
when the base point is moved along a path.

Recall that a path $\gamma$ is said to be regular with respect to a Morse
data $\Xi=(f,X)$ if the moduli spaces $\cM(\gamma,M)$ and $\cM(\gamma,x)$
are cutout transversely, i.e. if it is transverse to all the evaluation
maps $\cM(M,M)\xto{\ev_{-}}M$ and $\cM(M,x)\xto{\ev_{-}}M$ for
$x\in\Crit(f)$.

\begin{theorem}
  Let $M$ be a closed manifold and $\gamma_{\star}:[0,1]\to M$ a
  continuous path from a point $\star$ to another point $\star'$ in $M$.

  Given a regular stable Morse data $\Xi=(f,X,\star)$ such that
  $(f,X,\star')$ is also regular, and a smooth perturbation
  $\tilde{\gamma}_{\star}$ that is regular with respect to $\Xi$, there is
  group a isomorphism
  $\phi_{\gamma_{\star}*}:\pi_{1}(M,\star;\Xi)\to\pi_{1}(M,\star';\Xi)$
  that is defined out of bouncing trajectories, and does not depend on
  the perturbation $\tilde{\gamma}$.
\end{theorem}

\begin{proof}
  Given a Morse loop $\gamma$ based at $\star$,
  $\Theta(\tilde{\gamma}^{-1}_{\star}\ev_{+}(\gamma)\gamma_{\star})$ is
  well defined, and can be entirely described in terms of bouncing
  trajectories.

  It induces a map $\phi_{\gamma_{\star}*}$ at the $\pi_{1}$ level, and
  the statement is a consequence of the commutative diagram
  $$
  \begin{tikzcd}
    \pi_{1}(M,\star;\Xi)\ar[r,"\phi_{\gamma_{\star}*}"]
    \ar[d,"\ev_{*}"',"\sim"{rotate=90,anchor=north}]&
    \pi_{1}(M,\star';\Xi)
    \ar[d,"\ev_{*}"',"\sim"{rotate=90,anchor=north}]\\
    \pi_{1}^{(top)}(M,\star)\ar[r]&  \pi_{1}^{(top)}(M,\star')
  \end{tikzcd}.
  $$
\end{proof}

%%%%%%%%%%%%%%%
%\SkipToHere
%%%%%%%%%%%%%%

\nocite{*}
\printbibliography

\end{document}